\RequirePackage{fix-cm}
\documentclass[twocolumn]{svjour3}          
\smartqed  
\usepackage{graphicx}
\usepackage{amssymb}
\usepackage{color}
\usepackage{array}

\usepackage{amsmath}
\usepackage{stmaryrd}  
\usepackage{graphicx}
\usepackage{varwidth}
\usepackage{float}
\usepackage{array}
\usepackage{subfig}
\usepackage{grffile}

\newcommand{\off}[1]{}

\def\IR{\relax{\rm I\kern-.18em R}}
\def\p{\partial}
\def\<{\langle}
\def\>{\rangle}

\DeclareMathOperator{\Div}{div}

\graphicspath{{Eps/}}

\begin{document}

\title{Nonlinear Spectral Analysis via One-homogeneous Functionals - Overview and Future Prospects}


\author{Authors
}
\author{Guy Gilboa\thanks{Electrical Engineering Department, Technion – IIT} \and Michael Moeller\thanks{Department of Computer Science, Technische Universit\"at M\"unchen} \and Martin Burger\thanks{Institute for Computational and Applied Mathematics,
University of M\"unster}}
\authorrunning{G. Gilboa, M. Moeller and M. Burger} 

\date{Received: date / Accepted: date}


\maketitle
\bibliographystyle{plain}

\begin{abstract}
We present in this paper the motivation and theory of nonlinear spectral representations, based on convex regularizing functionals.
Some comparisons and analogies are drawn to the fields of signal processing, harmonic analysis and sparse representations.
The basic approach, main results and initial applications are shown. A discussion of open problems and future directions concludes this work.


\end{abstract}


\section{Introduction}

\off{
 PLAN

\begin{itemize}
\item Rationale - Fourier, SVD, eigenfunction analysis, linear algorithms.
\item Preliminaries
\item The Nonlinear eigenvalue problem
\item Spec TV
\item Spec 1-hom
\item The finite dimensional case
\item An eigenfunction filter view
\item Relations to
\begin{itemize}
\item Signal processing view (delta filters)
\item Wavelets -  general, thresholding of e.f's of l1. Haar wavelets as ef's of TV. Much more restricted, 1D decomp trial - less sparse, will contain more wavelets.
\item Linear ef problems
\item Sparse rep - a functional induces a dictionary (overcomplete?) by its eigenfunctions.
\item Cheeger sets?, ef's minimizes the per/area ratio ?
\end{itemize}
\item How to create eigenfunctions, check them, approximations etc.
\item Applications
\end{itemize}
}

Nonlinear variational methods have provided very powerful tools in the design and analysis of image processing and computer vision algorithms in recent decades.
In parallel, methods based on harmonic analysis, dictionary learning and sparse-representations, as well as spectral analysis of linear operators (such as the graph-Laplacian) have shown tremendous advances in processing highly complex signals such as natural images, 3D data and speech.
Recent studies now suggest variational methods can also be analyzed and understood through a nonlinear generalization of eigenvalue analysis, referred to
as nonlinear spectral methods. Most of the current knowledge is focused on one-homogeneous functionals, which will be the focus of this paper.

The motivation and interpretation of classical linear filtering strategies is closely linked to the eigendecomposition of linear operators. In this manuscript we will show that one can define a nonlinear spectral decomposition framework based on the gradient flow with respect to arbitrary convex 1-homogeneous functionals and obtain a remarkable number of analogies to linear filtering techniques. To closely link the proposed nonlinear spectral decomposition framework to the linear one, let us summarize earlier studies concerning the use of nonlinear eigenfunctions in the context of variational methods.

One notion which is very important is the concept of nonlinear eigenfunctions induced by convex functionals.
Given a convex functional $J(u)$ and its subgradient $\partial J(u)$, we refer to $u$ as an \emph{eigenfunction} if it admits the following eigenvalue problem:
\begin{equation}
\label{eq:ef_problem}
\lambda u  \in \p J(u),
\end{equation}
where $\lambda \in \mathbb{R}$  is the corresponding eigenvalue.

The analysis of eigenfunctions related to non-quadratic convex functionals was mainly concerned with the total variation (TV) regularization.
In the analysis of the variational TV denoising, i.e. the ROF model from \cite{rof92}, Meyer \cite{Meyer[1]} has shown an explicit solution for the case of a disk (an eigenfunction of TV), quantifying explicitly the loss of contrast
and advocating the use of $TV-G$ regularization.
Within the extensive studies of the TV-flow \cite{tvFlowAndrea2001,andreu2002some,bellettini2002total,Steidl,burger2007inverse_tvflow,discrete_tvflow_2012,tvf_giga2010} eigenfunctions of TV (referred to as \emph{calibrable sets}) were analyzed and explicit solutions were given for several cases of eigenfunction spatial settings. In \cite{iss} an explicit solution of a disk for the inverse-scale-space flow is presented, showing its instantaneous appearance at a precise time point related to its radius and height.

\begin{figure}[htb]
\begin{center}
\includegraphics[width=80mm]{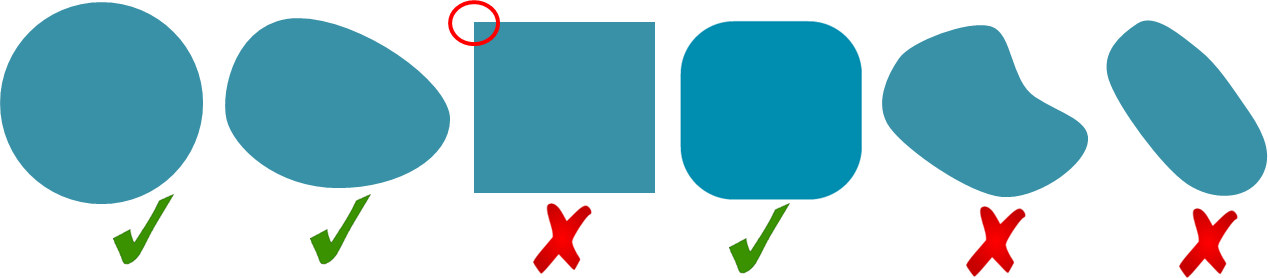}
\caption{Example of valid and non-valid shapes as nonlinear eigenfunctions with respect to the TV functional.
Smooth enough convex sets, which admit condition \eqref{eq:eigen_xi} (marked with a green check), are solutions of the eigenvalue problem \eqref{eq:ef_problem}.
Shapes with either too high curvature, not convex or too elongated (marked with red X) are not valid.
}
\label{fig:tv_efs}
\end{center}
\end{figure}

A highly notable contribution in \cite{bellettini2002total} is the precise geometric definition of convex sets which are eigenfunctions.
Let $\chi_C \in \mathbb{R}^2$ be a characteristic function, then it admits \eqref{eq:ef_problem}, with $J$ the TV functional, if
\begin{equation}
\label{eq:eigen_xi}
\textrm{ess} \sup_{q \in \p C}\kappa(q) \le \frac{P(C)}{|C|},
\end{equation}
where $C$ is convex, $\p C \in C^{1,1}$, $P(C)$ is the perimeter of $C$, $|C|$ is the area of $C$
and $\kappa$ is the curvature. In this case the eigenvalue is $\lambda = \frac{P(C)}{|C|}$. See Fig. \ref{fig:tv_efs} for some examples.
Until today there is no knowledge of additional (perhaps nonconvex sets) which admit the eigenvalue problem related to the TV functional.

In \cite{Muller_thesis,Benning_highOrderTV_2013,Poschl_Scherzer_TGV_1D_2013,Papafitsoros_Bredies_2013} eigenfunctions related to the total-generalized-variation (TGV) \cite{bredies_tgv_2010} and the infimal convolution total variation (ICTV) functional \cite{Chambolle[2]} are analyzed and their different reconstruction properties on particular eigenfunctions of the TGV are demonstrated theoretically as well as numerically.

Examples of certain eigenfunctions for different extensions of the TV to color images are given in \cite{collaborativeTV}.

In \cite{Steidl} Steidl et al. have shown the close relations, and equivalence in a 1D discrete
setting, of the Haar wavelets to both TV regularization and TV flow.
This was later developed for a 2D setting in \cite{Steidl_2D}. The connection between Haar wavelets and TV methods in 1D was made more precise in \cite{Benning_Burger_2013}, who indeed showed that the Haar wavelet basis is an orthogonal basis of eigenfunctions of the total variation (with appropriate definition at the domain boundary) - in this case the Rayleigh principle holds for the whole basis. In the field of morphological signal processing, nonlinear transforms were introduced in \cite{Dorst_Boomgaard_Slope_trans_1994,Kothe_Morph_SS_ECCV1996}.


\section{Spectral Representations}
\subsection{Scale Space Representation}
We will use the scale space evolution, which is straightforward, as the canonical case of spectral representation.
Consider a convex (absolutely) one-homogeneous functional $J$, i.e. a functional for which $J(\lambda u) = |\lambda| J(u)$ holds for all $\lambda \in \mathbb{R}$.

The \textit{scale space} or \textit{gradient flow} is
\begin{align}
\label{eq:scaleSpace}
\partial_t u(t) = -p(t), \ \ \ p(t) \in \partial J(u(t)), \ u(0)=f.
\end{align}
We refer the reader to \cite{wk_book,ak_book02} for an overview over scale space techniques.

\label{sec:SpectralScaleSpace}

A spectral representation based on the gradient flow formulation \eqref{eq:scaleSpace} was the first work towards defining a nonlinear spectral decomposition and has been conducted in \cite{Gilboa_SSVM_2013_SpecTV,Gilboa_spectv_SIAM_2014} for the case of $J$ being the TV regularization. In our conference paper \cite{spec_one_homog15}, we extended this notion to general 1-homogeneous functionals by observing, that the solution of the gradient flow can be computed explicitly for any 1-homogeneous $J$
 in the case of $f$ being an eigenfunction. For $\lambda f \in \partial J(f)$, the solution to \eqref{eq:scaleSpace} is given by
\begin{align}
\label{eq:linearBehavior}
u(t) = \left\{
\begin{array}{cc}
(1 -  t \lambda ) f & \text{ for } t\leq \frac{1}{\lambda}, \\
0 & \text{ else. }
\end{array}
\right.
\end{align}
Note that in linear spectral transformations such as Fourier or Wavelet based approaches, the input data being an eigenfunction leads to the energy of the spectral representation being concentrated at a single wavelength. To preserve this behavior for nonlinear spectral representations, the wavelength decomposition of the input data $f$ is defined by
\begin{equation}
\label{eq:phi}
\phi(t) = t \partial_{tt} u(t).
\end{equation}
Note that due to the piecewise linear behavior in \eqref{eq:linearBehavior}, the wavelength representation of an eigenfunction $f$ becomes $\phi(t) = \delta(t-\frac{1}{\lambda})f$, where $\delta$ denotes a Dirac delta distribution. The name \textit{wavelength} decomposition is natural because for $\lambda f \in \p J(f)$, $\|f\|=1$, one readily shows that $\lambda = J(f)$, which means that the eigenvalue $\lambda$ is corresponding to a generalized frequency. In analogy to the linear case, the inverse relation of a peak in $\phi$ appearing at $t = \frac{1}{\lambda}$ motivates the interpretation as a wavelength,
as discussed in more details in the following section.

For arbitrary input data $f$ one can reconstruct the original image by:
\begin{equation}
\label{eq:tv_recon}
f(x) = \int_0^\infty \phi(t;x) dt.
\end{equation}
Given a transfer function $H(t)\in \mathbb{R}$, image filtering can be performed by
\begin{equation}
\label{eq:tv_filt}
f_H(x) := \int_0^\infty H(t)\phi(t;x) dt.
\end{equation}

\section{Signal processing analogy}
Up until very recently, nonlinear filtering approaches such as \eqref{eq:scaleSpace} or related variational methods have been treated independently of the classical linear point of view of changing the representation of the input data, filtering the resulting representation and inverting the transform.
 In \cite{Gilboa_SSVM_2013_SpecTV,Gilboa_spectv_SIAM_2014} it was proposed to use \eqref{eq:scaleSpace} in the case of $J$ being the total variation to define a TV spectral representation of images that allows to extend the idea of filtering approaches from the linear to the nonlinear case.
This was later generalized to one-homogeneous functionals in \cite{spec_one_homog15}.

Classical Fourier filtering has some very convenient properties for analyzing and processing signals:
\begin{enumerate}
\item Processing is performed in the transform (frequency) domain by simple attenuation or amplification of desired frequencies
\item A straightforward way to visualize the frequency activity of a signal is through its spectrum plot. The spectral energy is preserved in
the frequency domain, through the well-known Parseval's identity.
\item The spectral decomposition corresponds to the coefficients representing the input signal in a new orthonormal basis.
\item Both, transform and inverse-transform are linear operations.
\end{enumerate}
In the nonlinear setting the first two characteristics are mostly preserved, orthogonality is still an open issue (where some results were obtained)
and linearity is certainly lost. In essence we obtain a \emph{nonlinear} forward transform and a \emph{linear} inverse transform.
Thus following the nonlinear decomposition,
filtering can be performed easily.

In addition, we gain edge-preservation and new scale features, which, unlike sines and cosines, are data-driven and are therefore highly
adapted to the image. Thus the filtering has far less tendency to create oscillations and artifacts.

Let us first derive the relation between Fourier and the eigenvalue problem \eqref{eq:ef_problem}.
For
$$J(u)=\frac{1}{2}\int |\nabla u(x)|^2 dx,$$
we get $-\Delta u \in \p J(u)$. Thus, with appropriate boundary conditions, Fourier frequencies are eigenfunctions, in the sense of \eqref{eq:ef_problem},
where, for a frequency $\omega$, we have the relation $\lambda = \omega^2$.
Other convex regularizing functionals, such as TV and TGV, can therefore be viewed as natural nonlinear generalizations.

A fundamental concept in linear filtering is \emph{ideal filters} \cite{Rabiner_dsp_book_1975}.
Such filters either retain or diminish completely frequencies within some range.
In a linear time (space) invariant system, a filter is fully determined by its transfer function
$H(\omega)$. The filtered response of a signal $f(x)$, with Fourier transform $F(\omega)$, filtered by a filter $H$ is
$$ f_H(x) := \mathcal{F}^{-1}\left(H(\omega)\cdot F(\omega) \right),$$
with $\mathcal{F}^{-1}$ the inverse Fourier transform.
For example, an ideal low-pass-filter retains all frequencies up to some cutoff frequency $\omega_c$.
Its transfer function is thus
\begin{align*}
H(\omega) = \left\{
\begin{array}{cc}
1 & \text{ for } 0 \le \omega \le \omega_c, \\
0 & \text{ else. }
\end{array}
\right.
\end{align*}
Viewing frequencies as eigenfunctions in the sense of \eqref{eq:ef_problem} one can
define generalizations of these notions.

\subsection{Nonlinear ideal filters}
The (ideal) low-pass-filter (LPF) can be defined by Eq. \eqref{eq:tv_filt} with $H(t)=1$ for $t \ge t_c$ and 0 otherwise, or
\begin{equation}
\label{eq:lpf}
LPF_{t_c}(f) := \int_{t_c}^\infty \phi(t;x) dt.
\end{equation}
Its complement, the (ideal) high-pass-filter (HPF), is defined by
\begin{equation}
\label{eq:hpf}
HPF_{t_c}(f) := \int_{0}^{t_c} \phi(t;x) dt.
\end{equation}
Similarly, band-(pass/stop)-filters are filters with low and high cut-off scale parameters ($t_1 < t_2$)
\begin{equation}
\label{eq:bpf}
BPF_{t_1,t_2}(f) := \int_{t_1}^{t_2} \phi(t;x) dt,
\end{equation}
\begin{equation}
\label{eq:bsf}
BSF_{t_1,t_2}(f) := \int_0^{t_1}\phi(t;x) dt + \int_{t_2}^\infty \phi(t;x) dt.
\end{equation}

For $f$ being a single eigenfunction with eigenvalue $\lambda_0 = \frac{1}{t_0}$ the above definitions
coincide with the linear definitions, where the eigenfunction is either completely preserved or
 completely diminished, depending on the cutoff eigenvalue(s) of the filter.

\begin{figure}[htb]
\begin{center}
\begin{tabular}{ cc }
\includegraphics[width=40mm]{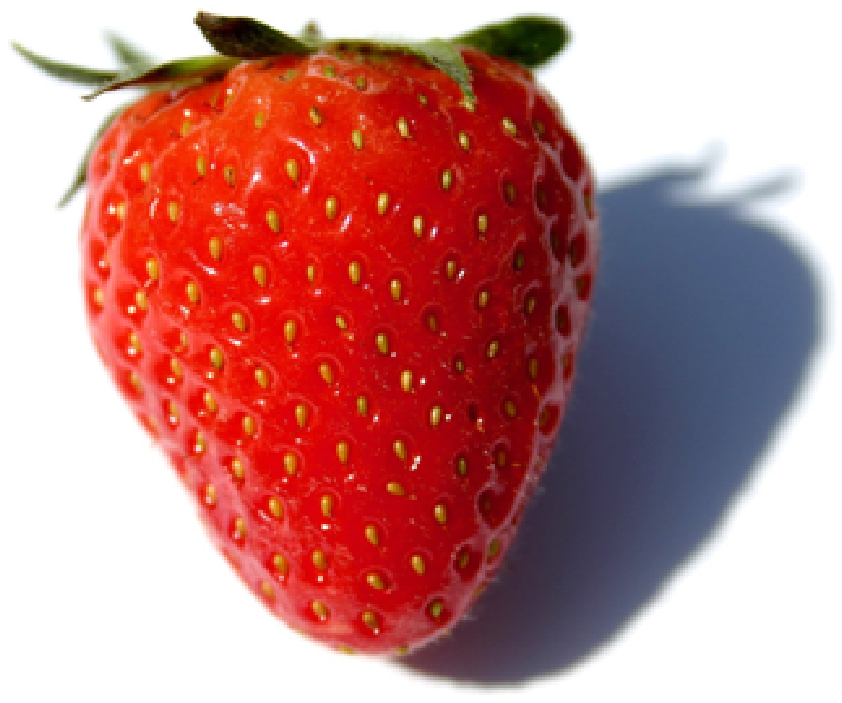}&
\includegraphics[width=40mm]{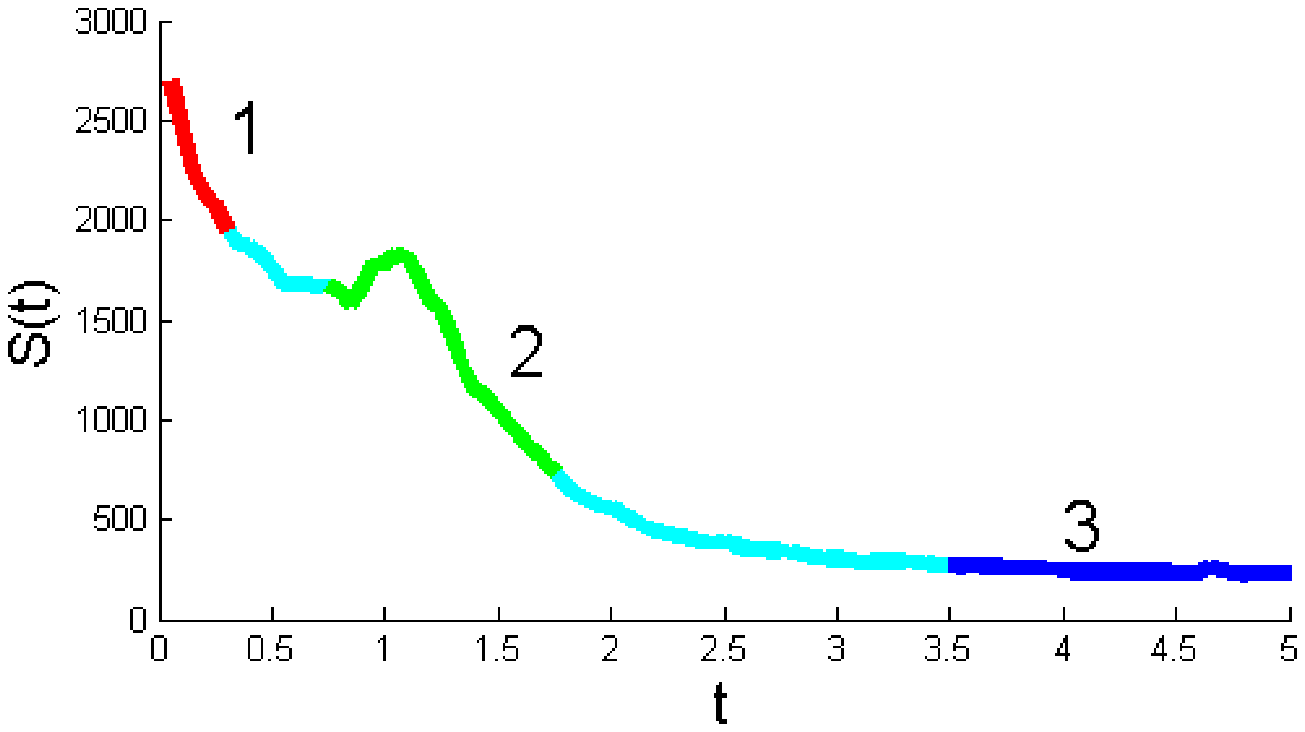}\\
Input $f$ &  $S(t)$\\
\includegraphics[width=40mm]{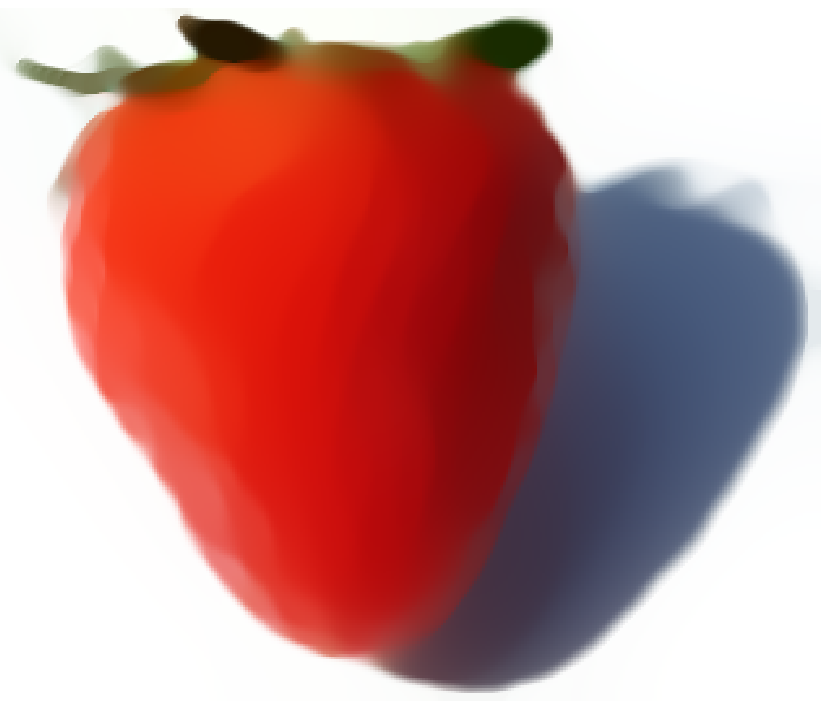}&
\includegraphics[width=40mm]{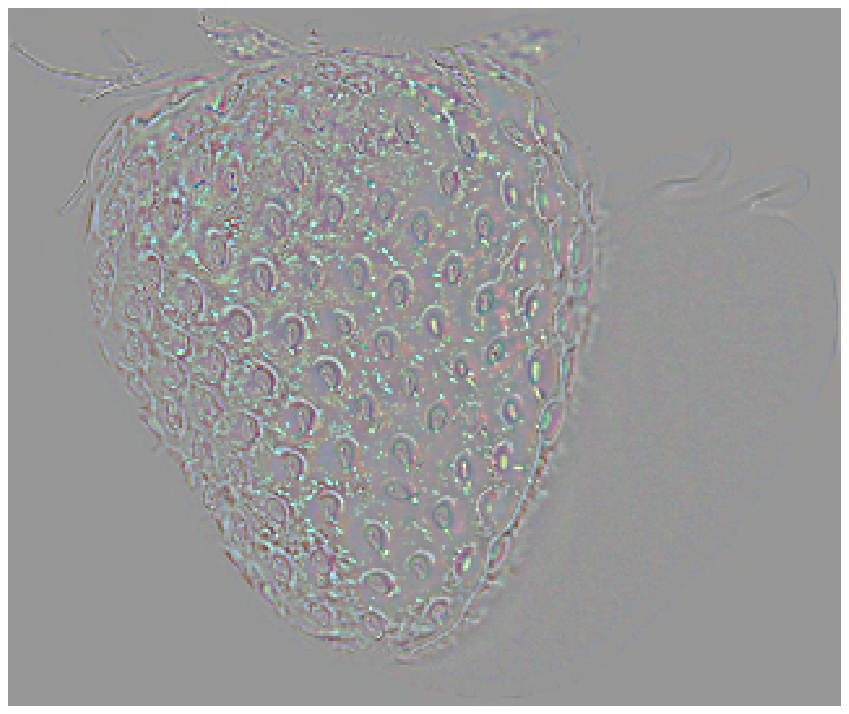}\\
Low-pass & High-pass\\
\includegraphics[width=40mm]{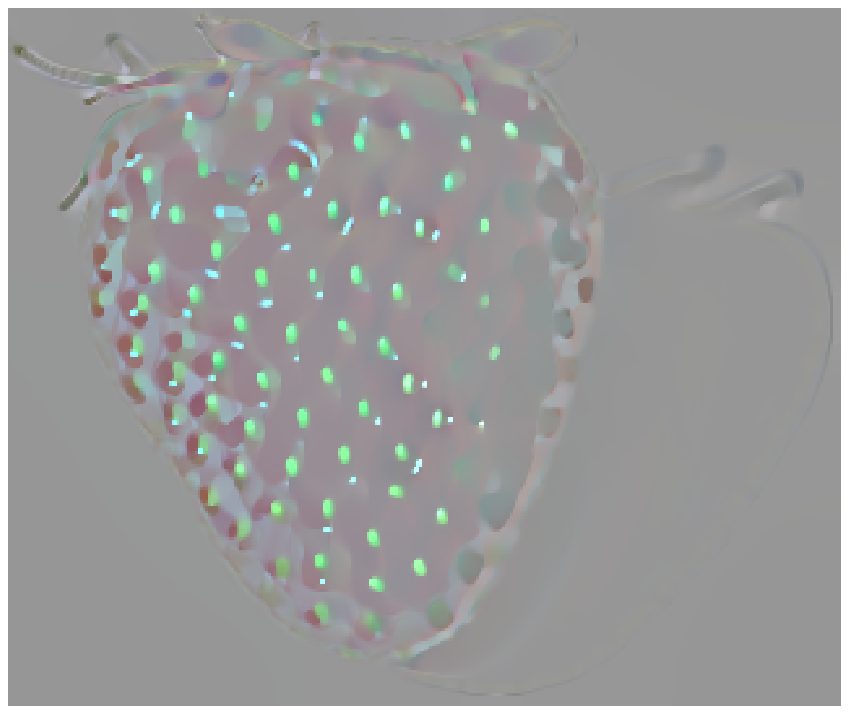}&
\includegraphics[width=40mm]{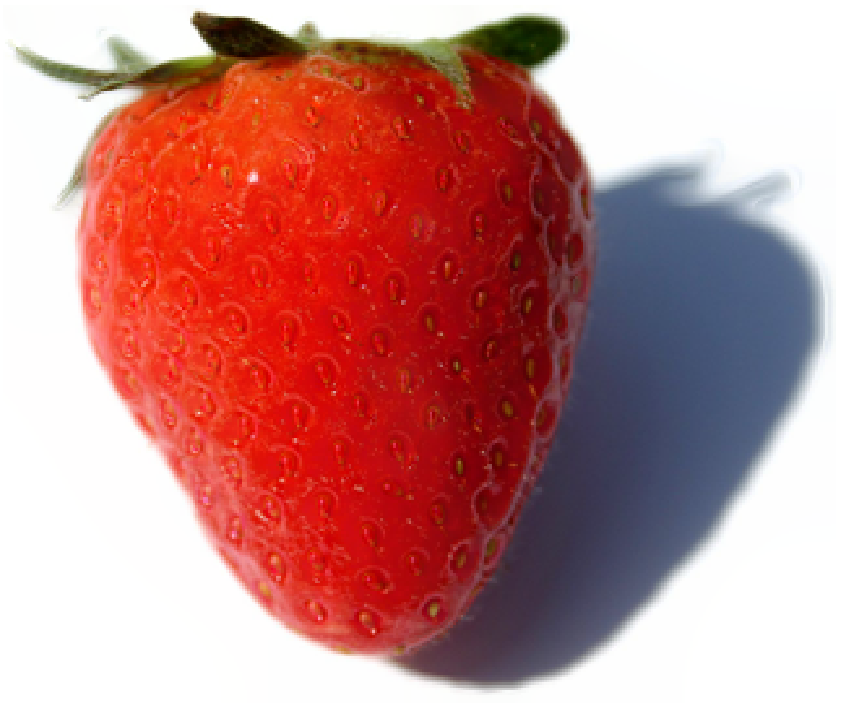}\\
Band-pass & Band-stop \\
\end{tabular}
\caption{Total-variation spectral filtering example. The input image (top left) is decomposed into its $\phi(t)$ components, the corresponding spectrum
$S(t)$ is on the top right. Integration of the $\phi$'s over the $t$ domains 1, 2 and 3 (top right) yields high-pass, band-pass and low-pass filters, respectively. The band-stop filter (bottom right) is the complement integration domain of region 2. Taken from \cite{spec_one_homog15}.}
\label{fig:strawberry}
\end{center}
\end{figure}

In Fig. \ref{fig:strawberry} an example of spectral TV processing is shown with the response of the four filters defined above in Eqs. \eqref{eq:lpf} through \eqref{eq:bsf}.

\subsection{Spectral response}

As in the linear case, it is very useful to measure in some sense the ``activity" at each frequency (scale).
This can help identify dominant scales and design better the filtering strategies (either manually or automatically).
Moreover, one can obtain a notion of the type of energy which is preserved in the new representation using some analog of
Parseval's identity.

In \cite{Gilboa_SSVM_2013_SpecTV,Gilboa_spectv_SIAM_2014} a $L^1$ type spectrum was suggested for the TV spectral framework
(without trying to relate to a Parseval rule),
\begin{equation}
\label{eq:S1}
S_1(t) := \|\phi(t;x)\|_{L^1(\Omega)} = \int_\Omega |\phi(t;x)|dx.
\end{equation}

In \cite{spec_one_homog15} the following definition was suggested,
\begin{equation} \label{eq:S2}
	S_2(t) = t \sqrt{ \frac{d^2}{dt^2}J(u(t)) } = \sqrt{\langle \phi(t), 2 t p(t)} \rangle.
\end{equation}
With this definition the following analogue of the Parseval identity was shown:
\begin{eqnarray}
	\Vert f \Vert^2 & = & - \int_0^\infty \frac{d}{dt} \Vert u(t) \Vert^2 ~dt \nonumber \\
    & = & 2 \int_0^\infty \langle p(t), u(t) \rangle ~dt \nonumber \\
    &= & 2 \int_0^\infty J(u(t))~dt \nonumber \\
    & = & \int_0^\infty S_2(t)^2  ~dt.	
\end{eqnarray}

In \cite{BEGM_1hom_SIAM_submitted} a third definition for the spectrum was suggested (which is simpler and admits Parseval),
\begin{equation} \label{eq:S3}
	S_3(t) =  \sqrt{\langle \phi(t), f } \rangle.
\end{equation}

It can be shown that a similar Parseval-type equality holds here
\begin{eqnarray}
	\Vert f \Vert^2 & = & \langle f,f \rangle =  \langle \int_0^\infty \phi(t)dt,f \rangle =  \int_0^\infty \langle \phi(t),f \rangle dt \nonumber \\
    & = & \int_0^\infty S_3(t)^2  ~dt.
\end{eqnarray}

As an overview example, Table \ref{tab:compare} summarizes the analogies of the Fourier and total variation based spectral transformation. There remain many open questions of further possible generalizations regarding Fourier-specific results, some of which are listed like the convolution theorem, Fourier duality property (where the transform and inverse transform can be interchanged up to a sign), the existence of phase in Fourier and more.

\begin{table*}[htbp]
\begin{center}
\begin{tabular}{|p{1.3in}||p{1.7in}|p{1.65in}|}
\hline
& {\bf TV transform} & {\bf Fourier transform}\\
\hline \hline
Transform & Gradient-flow: & $F(\omega) = \int f(x)e^{-i\omega x}dx$ \\
          & $\phi(t;x)=t \p_{tt}u$, $u_t \in -\p J(u)$. & \\
          & $u|_{t=0}=f$, (more representations in Table 1)& \\
Inverse transform & $f(x) = \int_0^\infty \phi(t;x)dt + \bar{f}$& $f(x) = \frac{1}{(2\pi)^n}\int F(\omega)e^{i \omega x}d\omega$\\
Eigenfunctions & $\lambda u \in \p J(u)$ & $e^{i\omega x}$ \\
Spectrum - amplitude & Three alternatives:  & $|F(\omega)|$\\
& $S_1(t)=\|\phi(t;x)\|_{L^1}$,   & \\
&$S_2(t)^2=\langle \phi(t), 2 t p(t)\rangle$, & \\
&$S_3(t)^2=\langle \phi(t), f \rangle$. & \\
Translation $f(x-a)$& $\phi(t;x-a)$& $F(\omega)e^{-ia \omega}$\\
Rotation  $f(Rx)$& $\phi(t;Rx)$, &$F(R\omega)$\\
Contrast change $a f(x)$& $\phi(t/a;x)$   & $a F(\omega)$\\
Spatial scaling  $f(ax)$& $a \phi(a t;a x)$ & $ \frac{1}{|a|}F(\frac{\omega}{a})$\\  
Linearity $f_1(x)+f_2(x)$ & Not in general & $F_1(\omega)+F_2(\omega)$\\
Parseval's identity  & $\int |f(x)|^2 dx = \int S_2(t)^2 dt = \int S_3(t)^2 dt$ & $\int |f(x)|^2 dx = \int |F(\xi)|^2 d\xi$, $\xi := \frac{\omega}{2\pi}$\\
Orthogonality  & $\langle \phi(t),u(t) \rangle = 0$ & $\langle e^{i\omega_1 x},e^{i\omega_2 x} \rangle = 0$, $\omega_1 \ne \omega_2$ \\
\hline
{\bf Open issues (some)}&  &\\
\hline
Duality & \center{--}  & $F(x) \leftrightarrow f(-\xi)$, $\xi := \frac{\omega}{2\pi}$ \\
Convolution  & \center{--} & $f_1(x) * f_2(x) \leftrightarrow F_1(\omega)F_2(\omega)$ \\
Spectrum - phase & \center{--} & $\tan^{-1}(\Im(F(\omega))/\Re(F(\omega)))$\\
\hline
\end{tabular}
\caption{ \label{tab:compare}
Some properties of the TV transform in $\IR^n$, compared to the Fourier transform.}
\end{center}
\end{table*}

\section{Decomposition into eigenfunctions}
As discussed in the section 3, item 3, one of the fundamental interpretations of linear spectral decompositions arises from it being the coefficients for representing the input data in a new basis. This basis is composed of the eigenvectors of the transform.  For example, in the case of the cosine transform, the input signal is represented by a linear combination of cosines with increasing frequencies. The corresponding spectral decomposition consists of the coefficients in this representation, which therefore admit an immediate interpretation.

Although the proposed nonlinear spectral decomposition does not immediately correspond to a change of basis anymore, it is interesting to see that the property of representing the input data as a linear combination of (generalized) eigenfunctions can be preserved: In \cite{spec_one_homog15} we showed that for the case of $J(u) = \|V u\|_1$ and $V$ being any orthonormal matrix, the solution of the scale space flow \eqref{eq:scaleSpace} meets
\begin{equation} \label{eq:gfWithOrthogonalL1Vu}
 Vu(t) = \text{sign}(\zeta) \; \max(|\zeta| - t, 0),
 \end{equation}
where $\zeta = Vf$ are the coefficients for representing $f$ in the orthonormal basis of $V$.
It is interesting to see that the subgradient $p(t)$ in \eqref{eq:scaleSpace} admits a componentwise representation of $Vp(t)$ as
\begin{equation} \label{eq:gfWithOrthogonalL1Vp}
 (Vp)_i(t) = \left\{ \begin{array}{ll}
\text{sign}(\zeta_i) & \text{ if } |\zeta_i| \geq t, \\
0 & \text{ else.}
\end{array} \right.
 \end{equation}
The latter shows that $p(t)$ can be represented by $V^T q(t)$ for some $q(t)$ which actually meets $q(t) \in \partial \|Vp(t)\|_1$. In a single equation this means $p(t) \in \partial J(p(t))$ and shows that the $p(t)$ arising from the scale space flow are eigenfunctions (up to normalization). Integrating the scale space flow equation \eqref{eq:scaleSpace} from zero to infinity and using that there are only finitely many times at which $p(t)$ changes, one can see that one can indeed represent $f$ as a linear combination of eigenfunctions. Our spectral decomposition $\phi(t)$ then corresponds to the \textit{change} of the eigenfunctions during the (piecewise) dynamics of the flow.

While the case of $J(u) = \|Vu\|_1$ for an orthogonal matrix $V$ is quite specific because it essentially recovers the linear spectral analysis exactly, we are currently investigating extensions of the above observations to more general types of regularizations in \cite{BEGM_1hom_SIAM_submitted}. In particular, the above results seem to remain valid in the case where $J(u)=\|Vu\|_1$ and $VV^*$ is diagonally dominant.

%
%

\section{Explicit TV Eigenfunctions in 1D}
Here we give an analytic expression of a large set of eigenfunctions of TV for $x \in \mathbb{R}$.
We will later see that Haar wavelets are a small subset of these, hence eigenfunctions are expected to represent signals more
concisely, with much fewer elements.

We give the presentation below in a somewhat informal manner, which we find more clear.
Note for instance that all the derivatives should be understood in a weak sense. A formal presentation of
the theory of TV eigenfunctions can be found e.g. in \cite{bellettini2002total}.

The TV functional can be expressed as
\begin{equation}
\label{eq:tv_z}
J_{TV}(u) = \sup_{\|z\|_{L^\infty(\Omega)\le 1}} \langle u, \Div z \rangle,
\end{equation}
with $\Div z \in L^2$ \cite{tvFlowAndrea2001,bellettini2002total}. Let $z_u$ be an argument admitting the supremum of \eqref{eq:tv_z}
then it immediately follows that $\Div z_u \in \p J(u)$: in the one homogeneous case
we need to show $\langle u, \Div z_u \rangle = J(u)$ which is given in \eqref{eq:tv_z};
and in addition that for any $v$ in the space we have $J(v)\ge \langle v,p \rangle$, $p\in \p J(u)$,
  $$J_{TV}(v) = \sup_{\|z\|_{L^\infty(\Omega)\le 1}} \langle v, \Div z \rangle \ge \langle v, \Div z_u \rangle.$$
From here on we will refer to $z_u$ simply as $z$.

To understand better what $z$ stands for, we can check the case of smooth $u$ and perform integration by parts in \eqref{eq:tv_z} 
to have
$$ J_{TV}(u) = \langle \nabla u, -z \rangle. $$
Then as $z$ also maximizes $\langle \nabla u, -z \rangle $, we can solve this pointwise, taking into
account that $|z(x)| \le 1$ and that the inner product of a vector is maximized for a vector at the same angle, to have
\begin{align}
z(x) = \left\{
\begin{array}{cc}
- \frac{\nabla u}{|\nabla u|} & \text{ for } \nabla u(x) \ne \bf{0}, \\
\in[-1,1] & \nabla u(x) = \bf{0}.
\end{array}
\right.
\end{align}

\begin{figure}[htb]
\begin{center}
\begin{tabular}{ ccc }
\includegraphics[width=23mm]{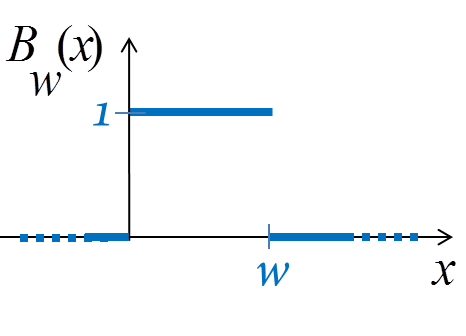}&
\includegraphics[width=23mm]{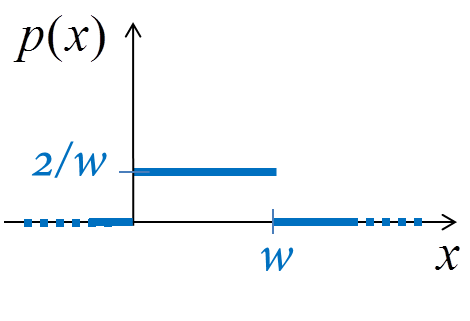}&
\includegraphics[width=23mm]{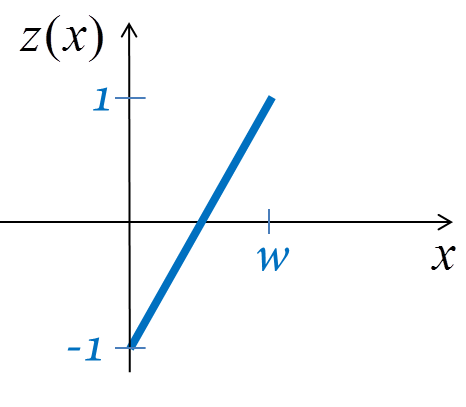}\\
$B_w(x)$ &  $p(x)$ & $z(x)$  \\
\end{tabular}
\caption{A classical single peak example of a TV eigenfunction in $\mathbb{R}$, $u(x)=B_w(x)$.}
\label{fig:single_peak}
\end{center}
\end{figure}

\subsection{Single peak}
Let us define the following function of a single unit peak of width $w$
\begin{align}
\label{eq:Bw}
B_w(x) = \left\{
\begin{array}{cc}
1 & x \in [0,w) \\
0 & \textrm{otherwise}.
\end{array}
\right.
\end{align}

Then it is well known (e.g. \cite{Meyer[1]}) that any function of the type $h\cdot B_w(x-x_0)$
is an eigenfunction of TV in $x\in\mathbb{R}$ with eigenvalue $\lambda = \frac{2}{hw}$.
Let us show how this is shown in terms of $z$. First, we should note that outside the support of the function, for
$x\in \Omega_0 = \mathbb{R} \setminus [w,w+x_0)$ we can use similar arguments of Lemma 6 in \cite{bellettini2002total} to have
$\Div z|_{x\in \Omega_0}=0$. It can be shown that $z$ defined by
\begin{align*}
z(x) = \frac{2(x-x_0)}{w}-1, \,\,\, x \in [x_0,x_0+w),
\end{align*}
and $\p_x z=0$ otherwise is a supremum in \eqref{eq:tv_z}, with the following subgradient
\begin{align}
p(x) = \left\{
\begin{array}{cc}
\p_x z = \frac{2}{w}, & x \in [0,w) \\
0, & \textrm{otherwise}.
\end{array}
\right.
\end{align}
Hence $p(x)=\lambda h  B_w(x-x_0)$ admits the eigenvalue problem \eqref{eq:ef_problem}.

\begin{figure}[htb]
\begin{center}
\begin{tabular}{ ccc }
\includegraphics[width=23mm]{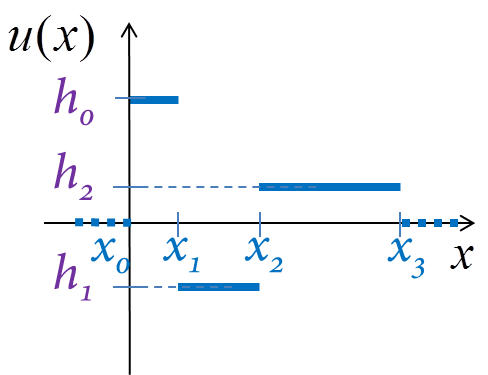}&
\includegraphics[width=23mm]{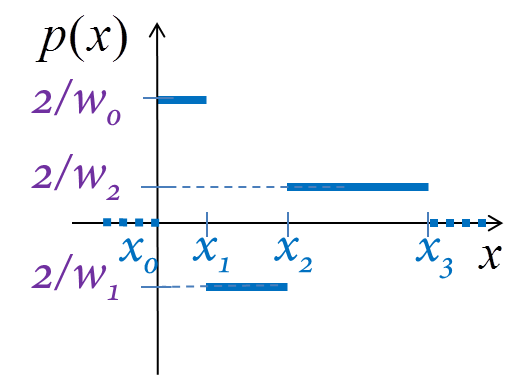}&
\includegraphics[width=23mm]{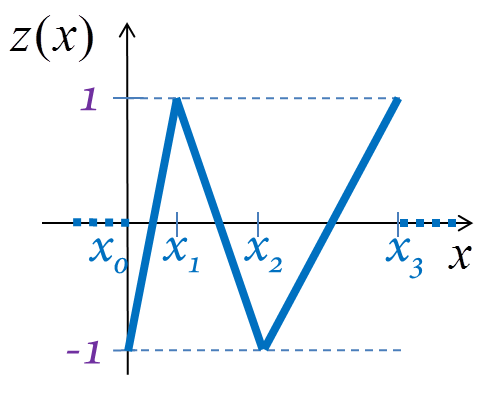}\\
$u(x)$ &  $p(x)$ & $z(x)$  \\
\end{tabular}
\caption{Illustration of Prop. \ref{prop:1d_ef}, a TV eigenfunction $u(x)$ in $\mathbb{R}$.}
\label{fig:u_ef}
\end{center}
\end{figure}

\begin{figure}[htb]
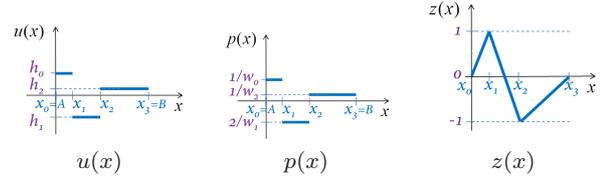

\begin{center}
\begin{tabular}{ ccc }
\includegraphics[width=23mm]{u_ef_bd.eps}&
\includegraphics[width=23mm]{u_ef_p_bd.eps}&
\includegraphics[width=23mm]{u_ef_z_bd.eps}\\
$u(x)$ &  $p(x)$ & $z(x)$  \\
\end{tabular}
\caption{A TV eigenfunction $u(x)$ in a bounded domain.}
\label{fig:u_ef_bd}
\end{center}
\end{figure}

\begin{figure}
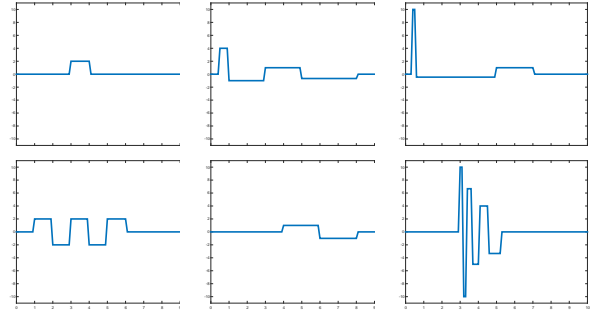

\begin{center}
	\tabcolsep0.3mm
	\begin{tabular}{ccc} 	
	\includegraphics[width=25mm]{efs_1d_1.eps} &
	\includegraphics[width=25mm]{efs_1d_2.eps} &
	\includegraphics[width=25mm]{efs_1d_3.eps} \\
	\includegraphics[width=25mm]{efs_1d_4.eps} &
	\includegraphics[width=25mm]{efs_1d_5.eps} &
	\includegraphics[width=25mm]{efs_1d_6.eps} \\
	\end{tabular}
	\caption{A few examples of functions meeting $f\in \partial J_{TV}(f)$ in $\mathbb{R}$.}
	\label{fig:tv_efs_1d}
\end{center}
\end{figure}

\subsection{Set of 1D eigenfunctions}
Generalizing this analysis one can construct for any eigenvalue $\lambda$
an infinite set of piece-wise constant eigenfunctions (with a compact support).
\begin{proposition}
\label{prop:1d_ef}
Let $-\infty < x_0 < x_1,..\,  < x_n < \infty$ be a set of $n+1$ points on the real line.
Let
\begin{equation}
\label{eq:u_ef}
u(x) =  \sum_{i=0}^{n-1} h_i  B_{w_i}(x-x_i),
\end{equation}
with $B_w(\cdot)$ defined in \eqref{eq:Bw}, $ w_i = x_{i+1}-x_i,$ and
\begin{equation}
\label{eq:h_ef}
h_i = \frac{2 (-1)^i}{\lambda w_i}.
\end{equation}
Then for $J$ the TV functional, $u(x)$ admits the eigenvalue problem \eqref{eq:ef_problem}.
\end{proposition}
\begin{proof}
One can construct the following $z$ in the shape of ``zigzag'' between $-1$ and $1$ at points $x_i$,
\begin{align*}
z(x) = (-1)^i\left( \frac{ 2(x-x_i)}{w_i}-1\right), \,\,\, x \in [x_i,x_{i+1}),
\end{align*}
and $\p_x z=0$ otherwise. In a similar manner to the single peak case we get the subgradient element in $\p J(u)$
\begin{align}
p(x) = \left\{
\begin{array}{cc}
\p_x z = (-1)^i \frac{2}{w_i}, & x \in [x_i,x_{i+1}) \\
0, & x \notin [x_0,x_n).
\end{array}
\right.
\end{align}
This yields $p(x) = \lambda u(x)$.
\end{proof}

See Fig. \ref{fig:u_ef}, \ref{fig:u_ef_bd}, and \ref{fig:tv_efs_1d} for some examples.

\subsubsection{The bounded domain case}
The unbounded domain is easier to analyze in some cases, however in practice we can implement only
signals with a bounded domain. We therefore give the formulation for $u \in \Omega=[A,B) \subset \mathbb{R}$.
Requiring $J(1)=\langle p,1 \rangle=0$ and $\langle u, p \rangle = \langle -\nabla u, z \rangle$
leads to the boundary condition $z|_{\partial \Omega}=0$.

Thus, on the boundaries we have $z=0$ with half the slope of the unbounded case (see Fig. \ref{fig:u_ef_bd}), all other
derivations are the same.
Setting $x_0 = A$, $x_n =B$ we get the solution of \eqref{eq:u_ef} with $h_i$ defined slightly differently as
\begin{equation}
\label{eq:h_ef_bounded}
h_i = \frac{2a_i (-1)^i}{\lambda w_i},
\end{equation}
where $a_i=\frac{1}{2}$ for $i\in \{1,n-1 \}$ and $a_i=1$ otherwise.
See a numerical convergence of $\lambda p$ to $u$ in Fig. \ref{fig:ef_numeric}.

{\bf Remark: } Note that if $u$ is an eigenfunction so is $-u$ so the formulas above are all valid also
with the opposite sign.

\section{Wavelets and Sparse Representation Analogy}
\subsection{Wavelets and hard thresholding}
As we have seen in section 4, the gradient flow with respect to regularizations of the form $J(u) = \|Vu\|_1$ has a closed form solution for any orthogonal matrix $V$.
From equation \eqref{eq:gfWithOrthogonalL1Vp} one deduces that
$$ \phi(t) = \sum_i \zeta_i \delta(t - |\zeta_i|) v_i, $$
where $v_i$ are the rows of the matrix $V$, and $\zeta_i = (Vf)_i$. In particular, using definition \eqref{eq:S3} we obtain
$$ (S_3(t))^2 = \sum_i (\zeta_i)^2 \delta(t - |\zeta_i|).$$
Peaks in the spectrum therefore occur ordered by the magnitude of the coefficients $\zeta_i = (Vf)_i$: The smaller $|\zeta_i|$ the earlier it appears in the wavelength representation $\phi$. Hence, applying an ideal low pass filter \eqref{eq:lpf} on the spectral representation with cutoff wavelength $t_c$ is exactly the same as hard thresholding by $t_c$, i.e. setting all coefficients $\zeta_i$ with magnitude less than $t_c$ to zero.


\subsubsection{Haar wavelets}
In general, wavelet functions are usually continuous and cannot represent discontinuities very well.
A special case are the Haar wavelets which are discontinuous and thus are expected to represent much better discontinuous signals.
The relation between Haar wavelets and TV regularization in one dimension has been investigated, most notably in \cite{Steidl}.
In higher dimensions, there is no straightforward analogy, in the case of isotropic TV, as it is not separable like wavelet decomposition
(thus one gets disk-like shapes as eigenfunctions as oppose to rectangles).

We will show here that even in the one-dimensional case TV eigenfunctions can represent signals in a much more concise manner.


Let $\psi_H(x)$ be a Haar mother wavelet function defined by
\begin{align}
\label{eq:Haar}
\psi_H(x) = \left\{
\begin{array}{cc}
1 & x \in [0,\frac{1}{2}), \\
-1 & x \in [\frac{1}{2},1), \\
0 & \textrm{otherwise}.
\end{array}
\right.
\end{align}
For any integer pair $n,k \in \mathbb{Z}$, a Haar function $\psi_H^{n,k}(x) \in \mathbb{R}$ is defined by
\begin{align}
\label{eq:Haar_nk}
\psi_H^{n,k}(x) = 2^{n/2}\psi_H(2^n x - k).
\end{align}

We now draw the straightforward relation between Haar wavelets and TV eigenfunctions.
\begin{proposition}
\label{prop:haar}
A Haar wavelet function $\psi_H^{n,k}$ is an eigenfunction of TV with eigenvalue $\lambda = 2^{(2+n/2)}$.
\end{proposition}
\begin{proof}
One can express the mother wavelet as
$$ \psi_H(x) = B_{\frac{1}{2}}(x)-B_{\frac{1}{2}}(x-\frac{1}{2}),$$
and in general, any Haar function as
$$ \psi_H^{n,k}(x) = h^n\left(B_{w^n}(x-x_0^{n,k})-B_{w^n}(x-x_0^{n,k}-w^n)\right),$$
with $h^n = 2^{(n/2)}$, $w^n=2^{-(n+1)}$ and $x_0^{n,k}=2^{-n}k$.
Thus, based on Prop. \ref{prop:1d_ef}, we have that for any $n,k \in \mathbb{Z}$, $\psi_H^{n,k}$ is
an eigenfunction of TV with $\lambda = \frac{2}{h^n w^n}= 2^{(2+n/2)}$.
\end{proof}

\begin{figure}
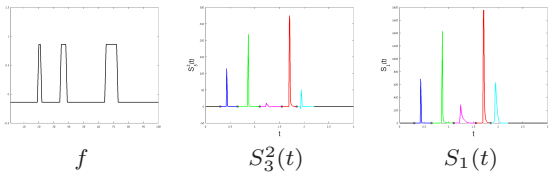

\begin{center}
	\tabcolsep0.3mm
	\begin{tabular}{ccc} 	
	\includegraphics[width=25mm]{3peaks_f.eps} &
	\includegraphics[width=25mm]{3peaks_S3.eps} &
	\includegraphics[width=25mm]{3peaks_S1.eps} \\
   $f$ & $S_3^2(t)$ & $S_1(t)$ \\
	\end{tabular}
	\caption{Decomposition example. }
	\label{fig:1d_decompA}
\end{center}
\end{figure}

\begin{figure}
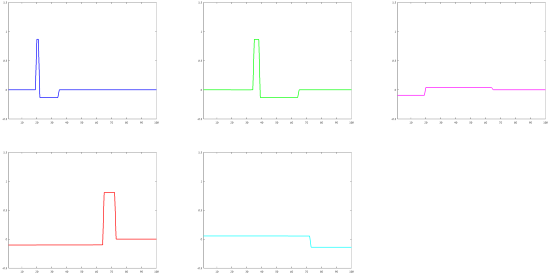

\begin{center}
	\tabcolsep0.3mm
	\begin{tabular}{ccc} 	
	\includegraphics[width=25mm]{3peaks_phi1.eps} &
	\includegraphics[width=25mm]{3peaks_phi2.eps} &
	\includegraphics[width=25mm]{3peaks_phi3.eps} \\
	\includegraphics[width=25mm]{3peaks_phi4.eps} &
	\includegraphics[width=25mm]{3peaks_phi5.eps} \\
	\end{tabular}
	\caption{Decomposing $f$ (Fig. \ref{fig:1d_decompA} left) into the 5 elements represented by the peaks in the spectrum using spectral TV. }
	\label{fig:1d_decompB}
\end{center}
\end{figure}

\begin{figure}
\begin{center}
	\tabcolsep0.3mm
	\includegraphics[width=75mm]{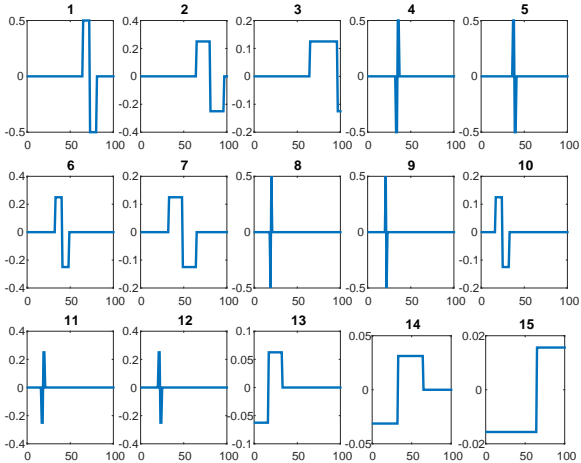}
	\caption{Decomposing $f$ (Fig. \ref{fig:1d_decompA} left) with Haar wavelets using 15 elements. }
	\label{fig:1d_decompC}
\end{center}
\end{figure}

\begin{figure}[htb]
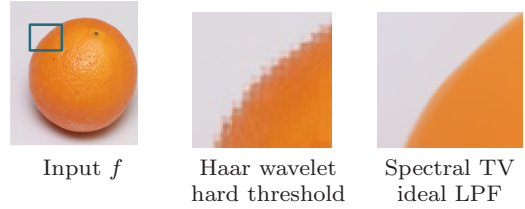

\begin{center}
\begin{tabular}{ ccc }
\includegraphics[width=20mm,height=20mm]{orange_f2_sq.eps}&
\includegraphics[width=20mm,height=18mm]{orange_haar2.eps}&
\includegraphics[width=19mm,height=18mm]{orange_spectv2.eps}\\
Input $f$ &  Haar wavelet & Spectral TV \\
& hard threshold & ideal LPF \\
\end{tabular}
\caption{Comparison of wavelet Haar hard-thresholding to spectral TV hard-thresholding (ideal LPF).
Although both representations can handle well discontinuities, the spectral TV representation is better adapted
to the image edges and produces less artifacts.
}
\label{fig:orange_haar}
\end{center}
\end{figure}

In Figs. \ref{fig:1d_decompA}, \ref{fig:1d_decompB} and \ref{fig:1d_decompC} we show the decomposition of a signal $f$ composed of 3 peaks of different width.
The TV spectral decomposition shows 5 numerical deltas (corresponding to the 5 elements depicted in  Fig. \ref{fig:1d_decompB}). On the other hand, Haar wavelet decomposition needs 15 elements to represent this signal, thus the representation is less sparse.
In the 2D case, Fig. \ref{fig:orange_haar}, we see the consequence of an inadequate representation, which is less adapted to the data, where Haar thresholding is compared to ideal TV low-pass filtering, depicting blocky artifacts in the Haar case.

\subsection{Sparse representation}
A functional induces a dictionary by its eigenfunctions. Let us formalize this notion by the following definition,
\begin{definition}[Eigenfunction Dictionary]
Let $\mathcal{D}_J$ be a dictionary of functions in a Banach space $\mathcal{X}$, with respect to a convex functional $J$, defined as the set of
all eigenfunctions,
$$\mathcal{D}_J := \bigcup_{\lambda \in \mathbb{R}}\{ u \in \mathcal{X}\, |\, \lambda u \in \p J(u)  \}.$$
\end{definition}

We can thus ask ourselves some questions such as
\begin{enumerate}
\item What functions can be reconstructed by a linear combination of functions in the dictionary $\mathcal{D}_J$?
\item Is $\mathcal{D}_J$ an overcomplete dictionary?
\item How many elements are needed to express some type of functions (or how sparse are they in this dictionary)?
\end{enumerate}
To most of these questions we do not have today a full answer.
Some immediate results can be derived by the analogy to Haar wavelets.

\begin{corollary}
For $J_{TV}$ being the TV functional, any $f\in L^2(\mathbb{R})$ can be approximated up to a desired error $\varepsilon$ by a finite linear combination of elements from  $\mathcal{D}_{J_{TV}}$.
Moreover, $\mathcal{D}_{J_{TV}}$ is an overcomplete dictionary (so the linear combination is not unique).
\end{corollary}
This follows from the fact that Haar wavelets are a subset of $\mathcal{D}_{J_{TV}}$. Thus for the first assertion one can use only Haar wavelets from the dictionary to reconstruct $f$
and use the wavelet reconstruction property \cite{Chui_wavelet_book}.
For the second assertion, any $B_w(x)$ can be decomposed into other elements in the dictionary, for instance, by $B_w(x) = \sum_{i=0}^{m-1} B_\frac{w}{m}(x-i\frac{w}{m})$, so there are many
(non-sparse) ways to reconstruct $f$.


\section{Rayleigh Quotients and SVD Decomposition}

In the classical case, the Rayleigh quotient is defined by
\begin{equation}
\label{eq:Rayleigh}
	R_M(v) := \frac{v^T M v}{v^T v},
\end{equation}
where for the real-valued case $M$ is a symmetric matrix and $v$ is a nonzero vector.
It can be shown that, for a given $M$, the Rayleigh quotient reaches its minimum value at $R_M(v)=\lambda_1$,
where $\lambda_1$ is the minimal eigenvalue of $M$ and $v=v_1$ is the corresponding eigenvector (and similarly for the maximum).

One can generalize this quotient to functionals, similar as in \cite{Benning_Burger_2013}, by
\begin{equation}
\label{eq:Rayleigh_J1}
	R_J(u) := \frac{J(u)^2}{\|u\|^2},
\end{equation}
where $\|\cdot \|$ is the $L^2$ norm. We restrict the problem to the non-degenerate case where
$J(u)>0$, that is, $u$ is not in the null-space of $J$.

To find a minimizer, an alternative formulation can be used (as in the classical case):
\begin{equation}
\label{eq:Rayleigh_J2}
	\min_u \{ J(u) \} \textrm{ s.t. } \|u\|^2=1.
\end{equation}
Using Lagrange multipliers, the problem can be recast as
$$	\min_u \{ J(u) - \frac{\lambda}{2}\left( \|u\|^2 - 1 \right) \},$$
with the optimality condition
$$  0 \in \partial J(u) - \lambda u,$$
which coincides with the eigenvalue problem \eqref{eq:ef_problem}. Indeed, one obtains the minimal nonzero eigenvalue as a minimum of the generalized Rayleigh quotient \eqref{eq:Rayleigh_J2}.

The work by Benning and Burger in \cite{Benning_Burger_2013} considers more general variational reconstruction problems involving a linear operator in the data fidelity term, i.e.,
$$\min_u \frac{1}{2}\|Au-f\|_2^2 + t \ J(u), $$
and generalizes equation \eqref{eq:ef_problem} to
$$ \lambda A^*A u \in \partial J(u), \qquad \|Au\|_2=1,$$
in which case $u$ is called a \textit{singular vector}. Particular emphasis is put on the \textit{ground states}
$$ u^0 = \text{arg}\min \limits_{ \substack{u \in \text{kern}(J)^\perp,\\ \|Au\|_2=1}} J(u) $$
for semi-norms $J$, which were proven to be singular vectors with the smallest possible singular value. Although the existence of a ground state (and hence the existence of singular vectors) is guaranteed for all reasonable $J$ in regularization methods, it was shown that the Rayleigh principle for higher singular values fails. As a consequence, determining or even showing the existence of a basis of singular vectors remains an open problem for general semi-norms $J$.

In the setting of nonlinear eigenfunctions for one-homogeneous functionals, the Rayleigh principle for the second eigenvalue (given the smallest eigenvalue $\lambda_1$ and ground state $u_1$)
\begin{equation}
\label{eq:Rayleigh_higher}
	\min_u \{ J(u) \} \textrm{ s.t. } \|u\|^2=1, \langle u, u_1 \rangle = 0 .
\end{equation}
With appropriate Lagrange multipliers $\lambda$ and $\mu$ we obtain the solution as a minimizer of
$$	\min_u \{ J(u) - \frac{\lambda}{2}\left( \|u\|^2 - 1 \right) + \mu \langle u, u_1 \rangle \},$$
with the optimality condition
$$ \lambda u - \mu u_1  = p  \in \partial J(u). $$
We oberve that we can only guarantee $u$ to be an eigenvector of $J$ if $\mu = 0$, which is not guaranteed in general. A scalar product with $u_1$ and the orthogonality constraint yields $\mu = - \langle p, u_1 \rangle$, which only needs to vanish if
$J(u) = \Vert u \Vert$, but not for general one-homogeneous $J$.

An example of a one-homogeneous functional failing to produce the Rayleigh principle for higher eigenvalues (which can be derived from the results in \cite{Benning_Burger_2013}) is given by
$J: \mathbb{R}^2 \rightarrow \mathbb{R}$,
$$ J(u) = \Vert D u \Vert_1, \qquad D = \left( \begin{array}{cc} 1 & - 2 {\epsilon} \\ 0 & \frac{1}{\epsilon} \end{array} \right), $$
for $0 < \epsilon < \frac{1}2$.
The ground state $u=(u^1,u^2)$ minimizes $\vert u^1 - 2\epsilon u^2\vert + \frac{1}\epsilon \vert u^2 \vert $ subject to $\Vert u \Vert_2 = 1$.
It is easy to see that $u=\pm (1,0)$ is the unique ground state, since by the normalization and the triangle inequality
\begin{align*}
 1 &= \Vert u \Vert_2 \leq \Vert u \Vert_1 \leq |u^1 - 2\epsilon u^2|+(1+2\epsilon )| u^2 |\\
 &\leq |u^1 - 2\epsilon u^2|+\frac{1}\epsilon | u^2 |,
\end{align*}
and the last inequality is sharp if and only if $u_2  =0$. Hence the only candidate $v$ being normalized and orthogonal to $u$ is given by $v = \pm (0,1)$.  Without restriction of generality we consider $v=(0,1)$. Now, $Dv = (-2\epsilon, \frac{1}\epsilon)$ and hence
$$ \partial J(v) = \{D^T (-1,1)\} = \{(-1,2\epsilon + \frac{1}\epsilon)\}, $$
which implies that there cannot be a $\lambda > 0$ with $\lambda v \in \partial J(v)$. A detailed characterization of functionals allowing for the Rayleigh principle for higher eigenvalues is still an open problems as well as the question whether there exists an orthogonal basis of eigenvectors in general (indeed this is the case for the above example as well).

\section{Cheeger sets and spectral clustering}

As mentioned in the introduction, there is a large literature on calibrable sets respectively Cheeger sets, whose characteristic function is an eigenfunction of the total variation functional, compare \eqref{eq:eigen_xi}. Particular interest is of course paid to the lowest eigenvalues and their eigenfunctions (ground states), which are related to solutions of the isoperimetric problem, easily seen from the relation $\lambda = \frac{P(C)}{|C|}$. Due to scaling invariance of $\lambda$, the shape of $C$ corresponding to the minimal eigenvalue can be determined by minimizing $P(C)$ subject to $|C|=1$. Hence, the ground states of the total variation functional are just the characteristic functions of balls in the isotropic setting. Interesting variants are obtained with different definitions of the total variation via anisotropic vector norms. E.g. if one uses the $\ell^1$-norm of the gradient in the definition, then it is a simple exercise to show that the ground state is given by characteristic functions of squares with sides parallel to the coordinate directions. In general, the ground state is the characteristic function of the Wulff shape related to the special vector norm used (cf. \cite{belloni2003isoperimetric,esedoglu2004decomposition})  and is hence a fingerprint of the used metric.

On bounded domains or in (finite) discrete versions of the total variation the situation differs, a trivial argument shows that the first eigenvalue is simply zero, with a constant given as the ground state. Since this is not interesting, it was even suggested to redefine the ground state as an eigenfunction of the smallest non-zero eigenvalue (cf. \cite{Benning_Burger_2013}). The more interesting quantity is the second eigenfunction, which is orthogonal to the trivial ground state (i.e. has zero mean). In this case the Rayleigh-principle always works (cf. \cite{Benning_Burger_2013}) and we obtain
\begin{equation}
	\lambda_2 = \inf_{u, \int u = 0} \frac{J(u)}{\Vert u \Vert_2}.
\end{equation}
Depending on the underlying domain on which total variation is considered (and again the precise definition of total variation) one obtains the minimizer as a function positive in one  part and negative in another part of the domain, usually with constant values in each. Hence, the computation of the eigenvalue can also be interpreted as some optimal cut through the domain.

The latter idea had a strong impact in data analysis (cf. \cite{szlam2009total,bresson2010total,bresson2012convergence,jost2014nonlinear,rangapuram2015constrained}), where the total variation is defined as a discrete functional on a graph, hence one obtains a relation to graph cuts and graph spectral clustering. For data given in form of a weighted graph, with weights $w_{ij}$ on an edge between two vertices $i,j$ in the vertex set ${\cal W}$, the total variation is defined as
\begin{equation}
	J(u) = \sum_{i,j \in {\cal V}} w_{ij} |u_i - u_j|
\end{equation}
for a function $u: {\cal V} \rightarrow \mathbb{R}$. Then a graph cut is obtained from a minimizing $u$ in
\begin{equation}
	\lambda_2 = \inf_{u, \sum u_i = 0} \frac{J(u)}{\Vert u \Vert_2}, \label{spectralcluster}
\end{equation}
i.e. an eigenfunction corresponding to the second eigenvalue. The cut directly yields a clustering into the two classes
\begin{equation}
	{\cal C}_+ = \{ i \in {\cal V}~|~u_i > 0 \}, \qquad {\cal C}_+ = \{ i \in {\cal V}~|~u_i < 0 \}.
\end{equation}
Note that in this formulation the graph cut technique is completely analogous to the classical spectral clustering technique based on second eigenvalues of the graph Laplacian (cf. \cite{von2007tutorial}). In the setting \eqref{spectralcluster}, the spectral clustering based on the graph Laplacian can be incorporated by using the functional
 \begin{equation}
	J(u) = \sqrt{ \sum_{i,j \in {\cal V}} w_{ij} |u_i - u_j|^2}.
\end{equation}

Some improvements in spectral clustering are obtained if the normalization in the eigenvalue problem is not done in the Euclidean norm (or a weighted variant thereof), but some $\ell^p$-norm, i.e.,
\begin{equation}
	\lambda_2 = \inf_{u, \sum u_i = 0} \frac{J(u)}{\Vert u \Vert_p}. \label{spectralclusterp}
\end{equation}
We refer to \cite{jost2014nonlinear} for an overview, again the case $p=1$ has received particular attention.
In the case of graph-total-variation, see a recent initial analysis and a method to construct certain type of concentric eigenfunctions in \cite{spec_nltv15}.
A thorough theoretical study of eigenvalue problems and spectral representations for such generalized eigenvalue problems is still an open problem, as for any non-Hilbert space norm. The corresponding gradient flow becomes a doubly nonlinear evolution equation of the form
\begin{equation}
	0 \in  \partial \Vert \partial_t u \Vert_p + \partial J(u) .
\end{equation}
The gradient flow for $p=1$ and total variation in the continuum is related to the $L^1$-TV model, which appears to have interesting multiscale properties (cf. \cite{yin2007total}).

A challenging problem is the numerical computation of eigenvalues and eigenfunctions in the nonlinear setup, which is the case also in the general case beyond the spectral clustering application. Methods reminiscent of the classical inverse power method in eigenvalue computation (cf. e.g. \cite{hein2010inverse}) have been proposed and are used with some success, but it is difficult to obtain and guarantee convergence to the smallest or second eigenvalue in general.

\begin{figure}[htb]
\begin{center}
\begin{tabular}{ cc }
\includegraphics[width=56mm]{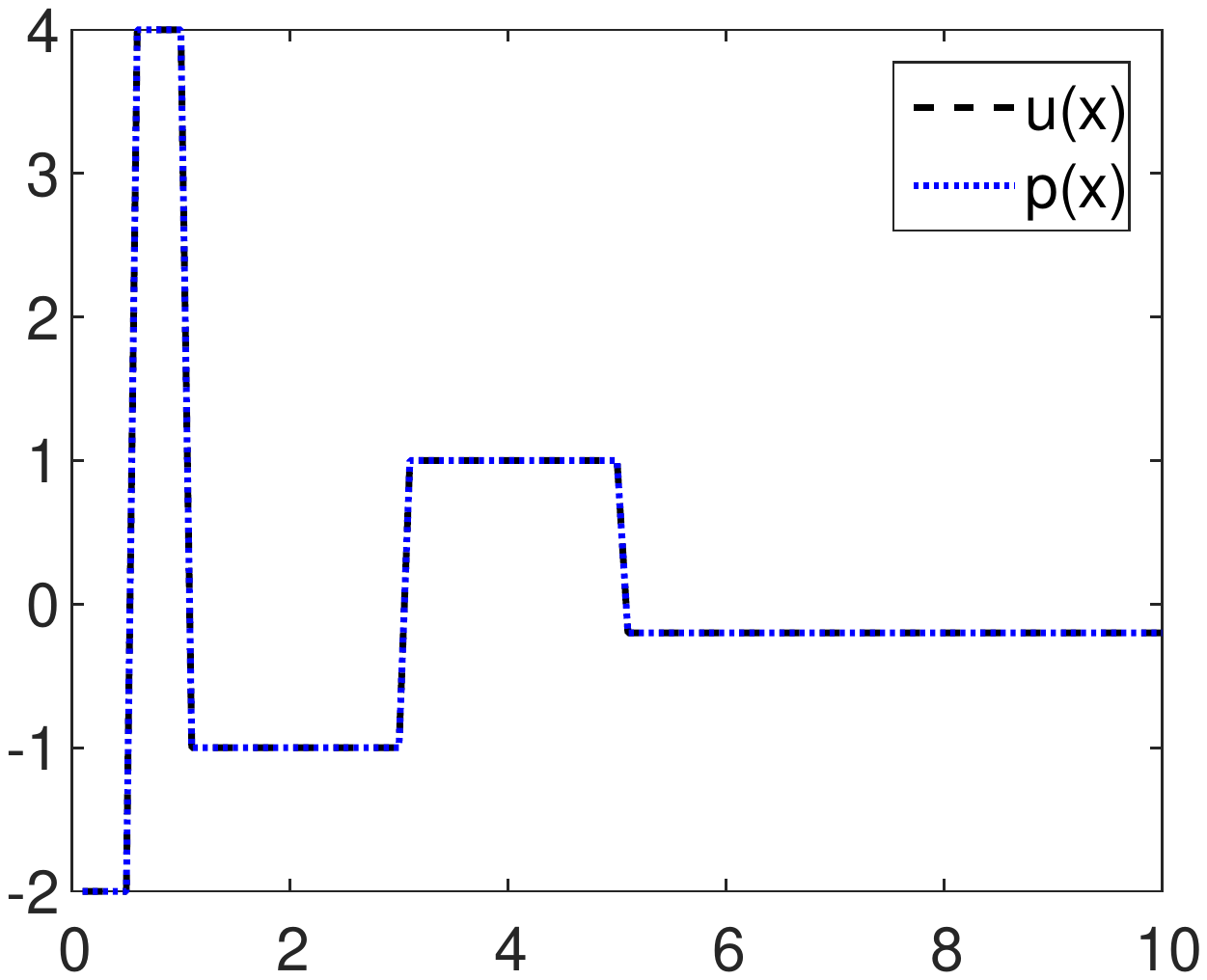}\\
$u(x)$ and $p(x)$ at convergence\\
\includegraphics[width=60mm]{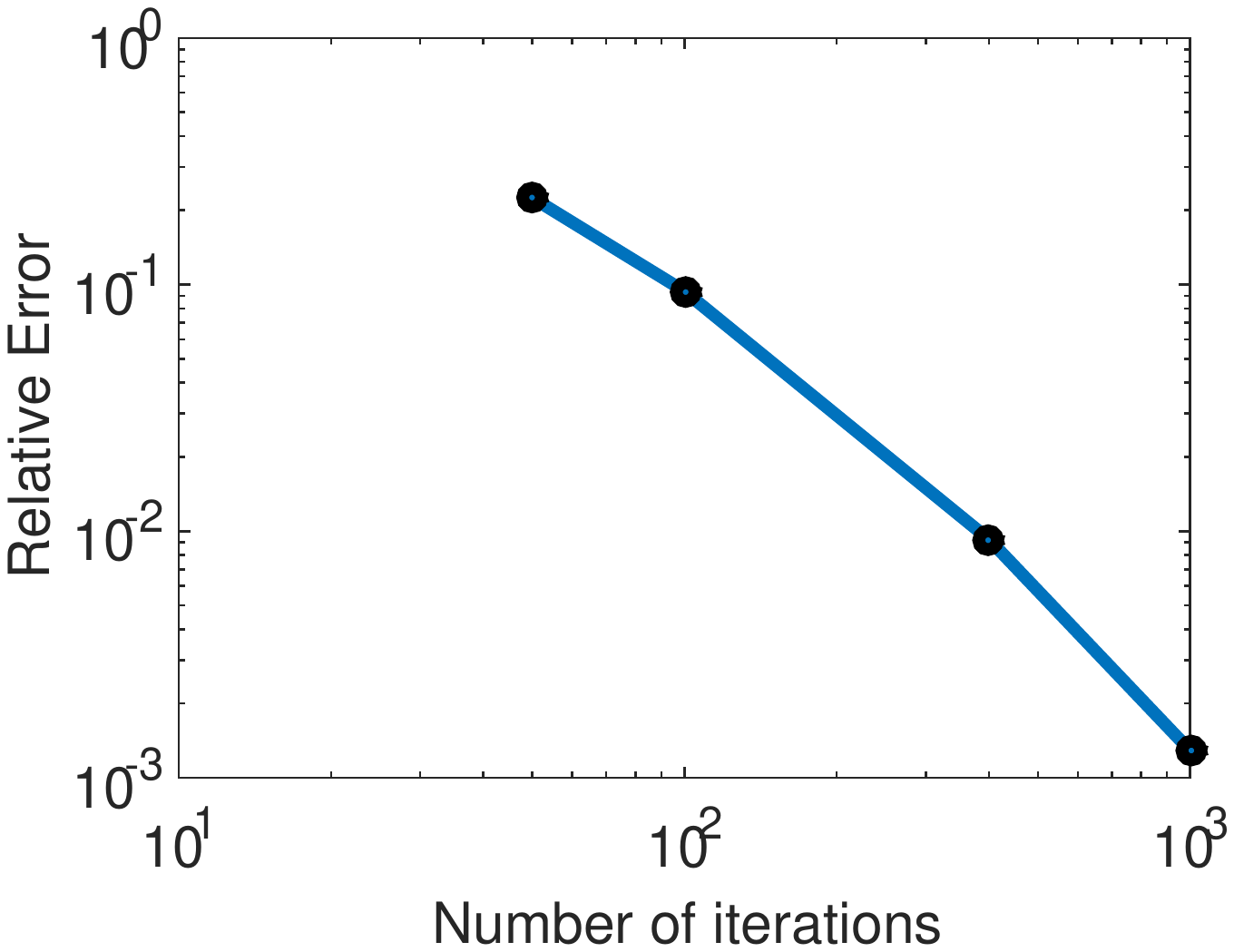}\\
Relative absolute error to the solution at\\
convergence as a function of iterations.\\
\includegraphics[width=60mm]{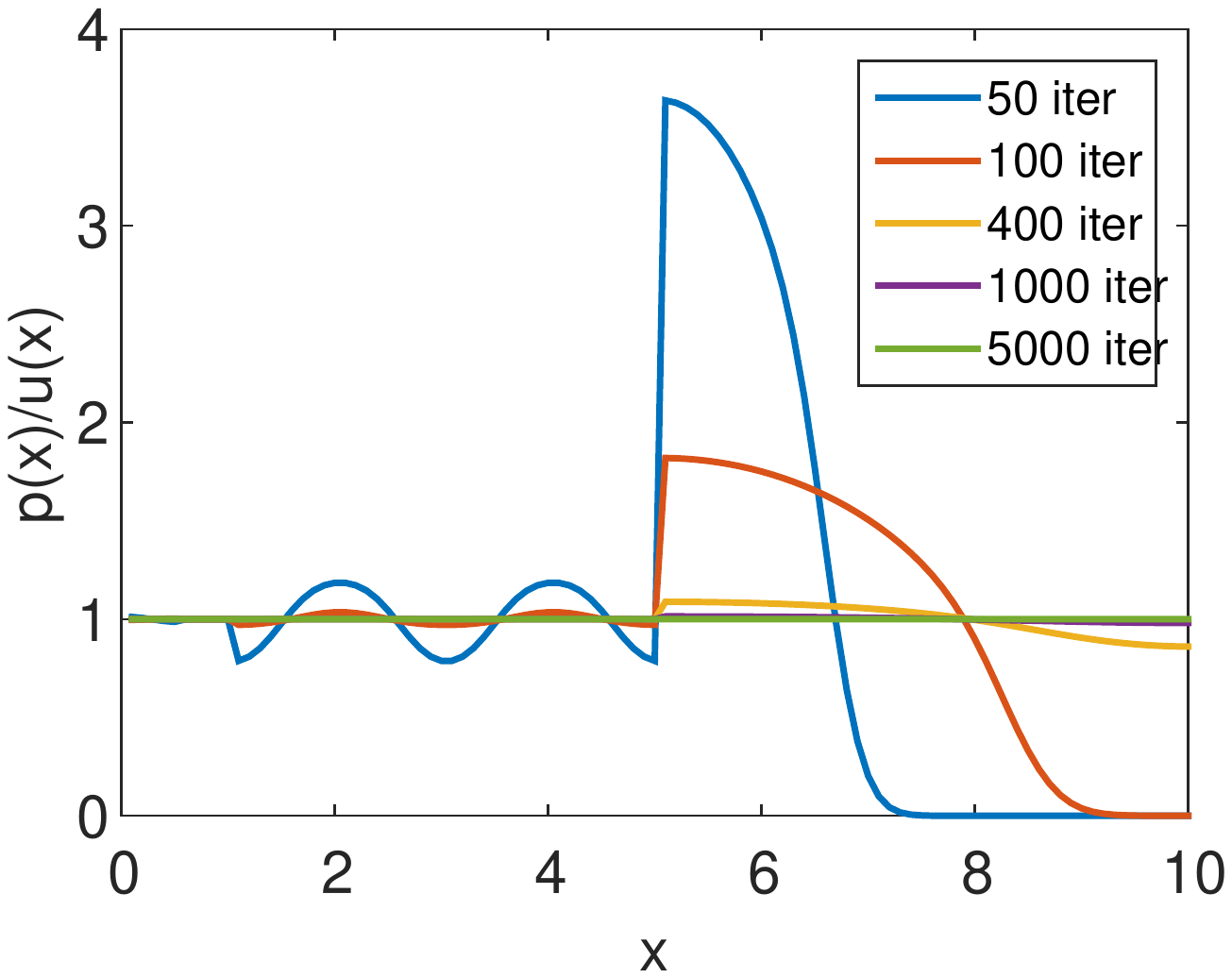}\\
 Ratio $\frac{p(x)}{u(x)}$ for different number of iterations\\
\end{tabular}
\caption{Numerical implementation of a 1D TV eigenfunction, $\lambda=1$,
Chambolle-Pock scheme \cite{Chambolle-Pock2011}. The ratio $\frac{p(x)}{u(x)}=\lambda$ should converge to 1.
It takes about 5,000 iterations for the numerical scheme to fully convergence. Note that
the error is data-dependent, where larger flat regions converge much slower.
}
\label{fig:ef_numeric}
\end{center}
\end{figure}

\section{Numerical Implementation Issues}
Here we give the gradient flow approximation.
We would like to approximate $ u_t = -p(u) $.
We use the Moreau-Yosida approximation \cite{Moreau1965} for gradient flow. The implicit discrete evolution (which is unconditionally stable in $dt$) is
$$ u(n+1) = u(n) - dtp(u(n+1)).$$
We can write the above expression as
$$ u(n+1) - u(n) + dtp(u(n+1)) = 0,$$
and see it coincides with the Euler-Lagrange of the following minimization:
$$ E(u,u(n)) = J(u) + \frac{1}{2dt}\|u - u(n)\|_{l^2}^2, $$
where $u(n+1)$ is the minimizer $u$ of $E(u,u(n))$ and $u(n)$ is fixed.
We now have a standard convex variational problem which can be solved using various algorithms (e.g. \cite{darbon.06.jmiv,Chambolle[1],SplitBregman2009,Pock_min_mumford_shah_ICCV2009,Chambolle-Pock2011}).

To approximate the second time derivative $u_{tt}(t)$ we store in memory 3 consecutive time steps of $u$ and use the standard central scheme:
$$ D^2 u(n) = \frac{u(n-1)+u(n+1)-2u(n)}{(dt)^2}, $$
with $n=1,2,..$, $u(0)=f$, $D^2 u(0)=0$. The time $t$ is discretized as $t(n) = ndt$. Therefore:
\begin{equation}
\label{eq:phi_n}
\phi(n) = D^2 u(n) t(n) = \frac{n}{dt}\left(u(n-1)+u(n+1)-2u(n)\right),
\end{equation}
and for instance the spectrum $S_1$ is
\begin{equation}
\label{eq:S_n}
S_1(n) = \sum_{i \in \mathcal{N}}  |\phi_i(n)|,
\end{equation}
where $i$ is a pixel in the image and $\mathcal{N}$ is the image domain.

In Fig. \ref{fig:ef_numeric} we show that $u$ is a numerical eigenfunction where $p$ converges to $\lambda u$.
In this case we compute $p$ by the variational model above (with a small $dt$) using the scheme of \cite{Chambolle-Pock2011}.
It is shown that such schemes are far better for oscillating signals (works well for denoising), but when there are large
flat regions the local magnitude of $p$ depends on the support of that region and convergence is considerably slower
(can be orders of magnitude). Thus alternative ways are desired to solve such cases more efficiently.

\begin{table*}[htbp]
\begin{center}
\begin{tabular}{|p{1.9in}||p{1.0in}|p{1.0in}|}
\hline
 & {\bf Wavelength } & {\bf Frequency }\\
\hline \hline
{\bf Gradient flow} & $\phi(t) = t \p_{tt} u(t) $ & $\psi(s) = \frac{1}{s^2}\phi(\frac{1}{s})$ \\
$\p_t u(t) \in -\p J(u(t)),$ $u|_{t=0}=f$. & & \\ \hline
{\bf Variational minimization} & $\phi(t) = t \p_{tt} u(t) $ & $\psi(s) = \frac{1}{s^2}\phi(\frac{1}{s})$ \\
$t J(u) + \frac{1}{2}\|f-u\|_2^2$. & & \\ \hline
{\bf Inverse-scale-space} & $\phi(t) = \frac{1}{t^2}\psi(\frac{1}{t})$ & $\psi(s)=\p_s u(s)$ \\
$\p_s p(s) = f-u(s),$ $p(s) \in \p J(u(s))$. & & \\
\hline
\end{tabular}
\caption{ \label{tab:rep}
Six spectral representations of one-homogeneous functionals.}
\end{center}
\end{table*}

\section{Extensions to other nonlinear decomposition techniques}
As presented in
\cite{spec_one_homog15} the general framework of nonlinear spectral decompositions via one-homogeneous regularization functionals does not have to be based on the gradient flow equation \eqref{eq:scaleSpace}, but may as well be defined via a variational method of the form
\begin{align}
\label{eq:vm}
 u(t) = \arg \min_u \frac{1}{2}\|u-f\|_2^2 + t J(u),
\end{align}
or an inverse scale space flow of the form
\begin{align}
\label{eq:iss}
 \partial_s p(s) = f - u(s), \qquad p(s) \in \partial J(u(s)).
\end{align}
Since the variational method shows exactly the same behavior as the scale space flow for the input data $f$ being an eigenvector of $J$, the definitions of the spectral decomposition coincides with the one of the scale the flow. The behavior of the inverse scale space flow on the other hand is different in two respects. Firstly, it start at $u(0)$ being the projection of $f$ onto the kernel of $J$ and converges to $u(s)=f$ for sufficiently large $s$. Secondly, the dynamics for $f$ being an eigenfunction of $J$ is piecewise constant in $u(s)$ such that only a single derivative is needed to obtain a peak in the spectral decomposition. The behavior on eigenfunctions furthermore yields relation of $s = \frac{1}{t}$ when comparing the spectral representation obtained by the variational or scale space method with the one obtained by the gradient flow.

In analogy to the linear spectral analysis it is natural to call the spectral representation of the inverse scale space flow a \textit{frequency representation} opposed to the \textit{wavelength representation} of the variational and scale space methods.

To be able to obtain frequency and wavelength representations for all three types of methods, we make the convention that a filter $H(t)$ acting on the wavelength representation $\phi$ should be equivalent to the filter $H(1/t)$ acting the frequency representation $\psi$, i.e.
$$ \int_0^\infty \phi(t) H(t) ~dt = \int_0^\infty \psi(t) H(1/t) ~dt. $$
By a change of variables one deduces that
$$ \psi(t) = \frac{1}{t^2}\phi(\frac{1}{t}) \text{ respectively }  \phi(t) = \frac{1}{t^2}\psi(\frac{1}{t})$$
are the appropriate conversion formulas to switch from the frequency to the wavelength representation and vice versa. Table \ref{tab:rep} gives an overview over the three nonlinear decomposition methods and their spectral representation in terms of frequencies and wavelength.

It can be verified that in particular cases such as regularizations of the form $J(u) = \|Vu\|_1$ for an orthonormal matrix $V$, or for the data $f$ being an eigenfunction of $J$, the gradient flow, the variational method and the inverse scale space flow yield exactly the same spectral decompositions $\phi$ and $\psi$. In \cite{BEGM_1hom_SIAM_submitted} we are investigating more general equivalence results and we refer the reader to \cite{Eckardt_Burger_Bach_thesis_2014} for a numerical comparison of the above approaches. A precise theory of the differences of the three approaches for a general one-homogeneous $J$, however, remains an open question.

\section{Preliminary applications}
\subsection{Filter design for denoising}
{One particular application of the spectral decomposition framework could be the design of filters in denoising applications. It was shown in \cite[Theorem 2.5]{Gilboa_spectv_SIAM_2014} that the popular denoising strategy of evolving the scale space flow \eqref{eq:scaleSpace} for some time $t_1$, is equivalent to the particular filter
$$ H(t) = \left\{\begin{array}{ll}
0 & \text{ for } 0 \leq t \leq t_1 \\
\frac{t-t_1}{t} & \text{ for } t_1 \leq t \leq \infty
\end{array} \right. $$
in the framework of equation \eqref{eq:tv_filt}. The latter naturally raises the question if this particular choice of filter is optimal for practical applications. While in case of a perfect separation of signal and noise an ideal low pass filter \eqref{eq:lpf} may allow a perfect reconstruction (as illustrated on synthetic data in figure \ref{fig:idealLowPass}), the spectral decomposition of natural noisy images will often contain noise as well as texture in high frequency components.
\begin{figure*}
\includegraphics[width=0.16\textwidth]{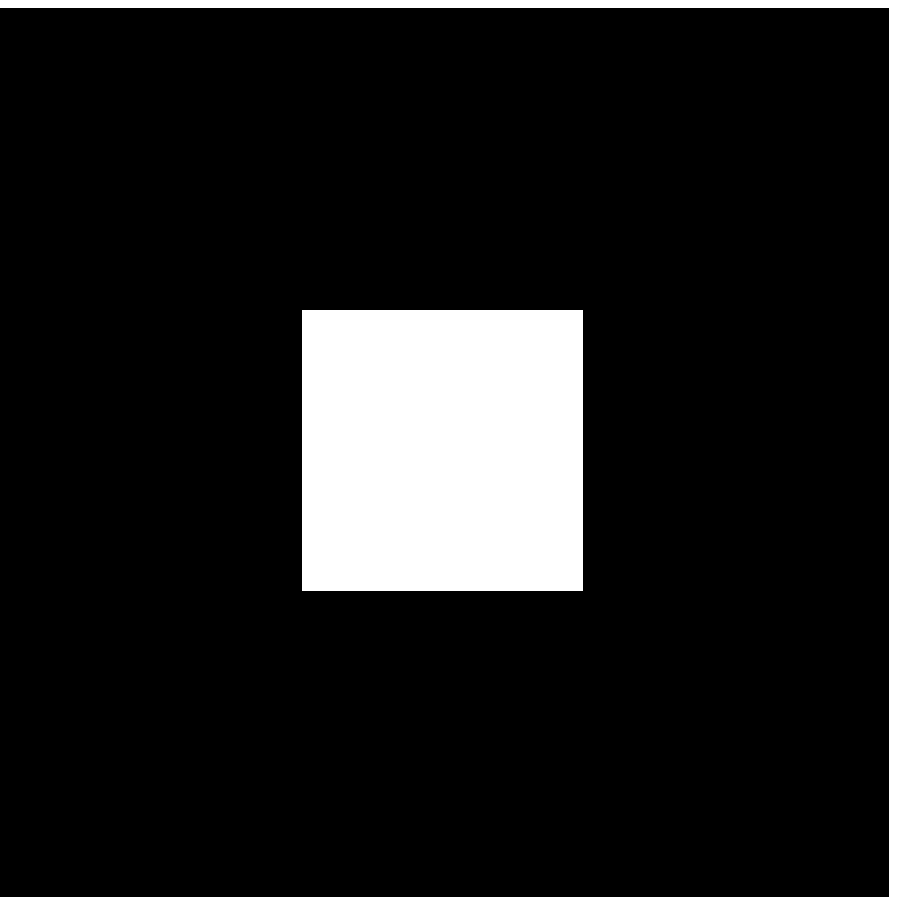}
\includegraphics[width=0.21\textwidth]{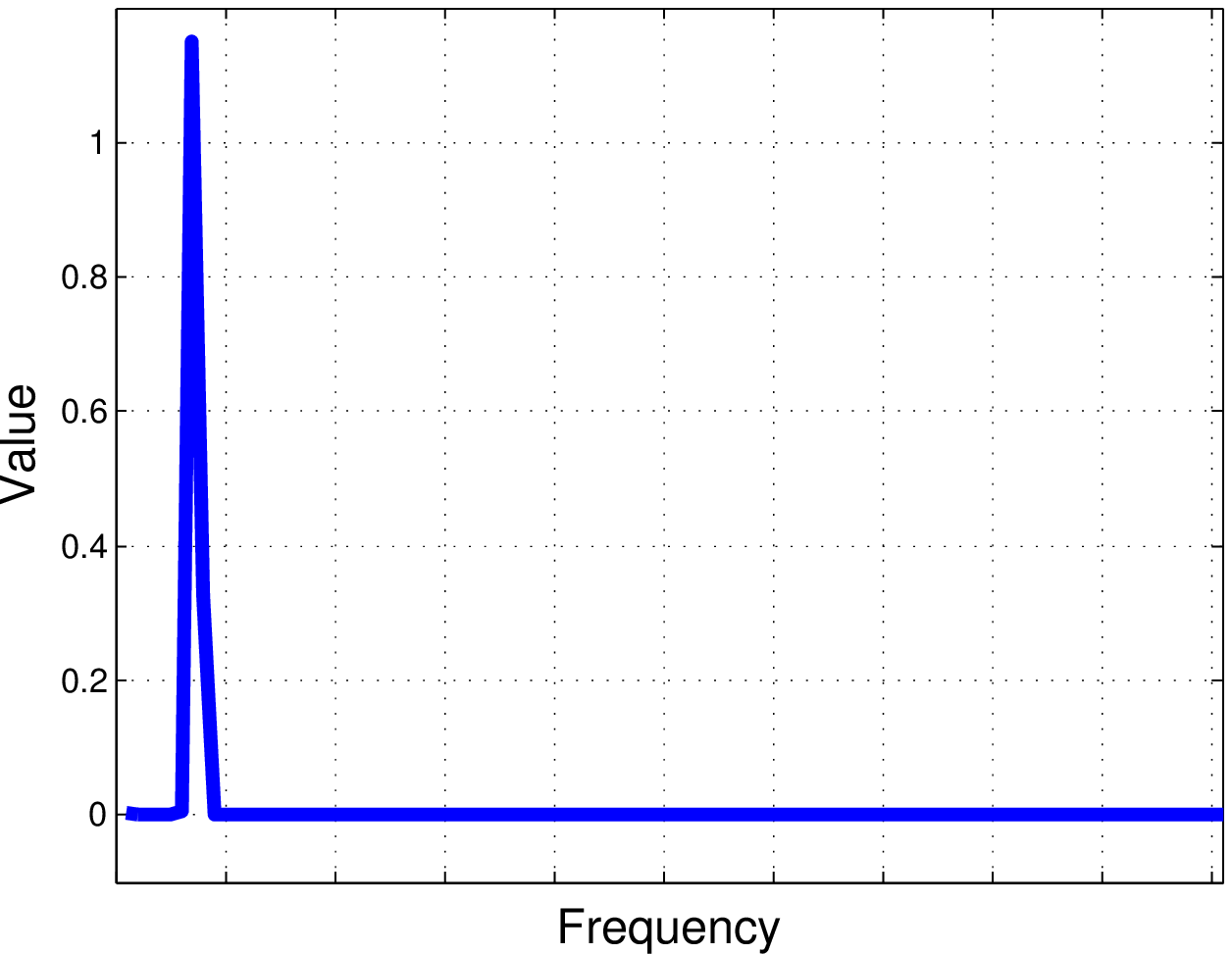}
\includegraphics[width=0.16\textwidth]{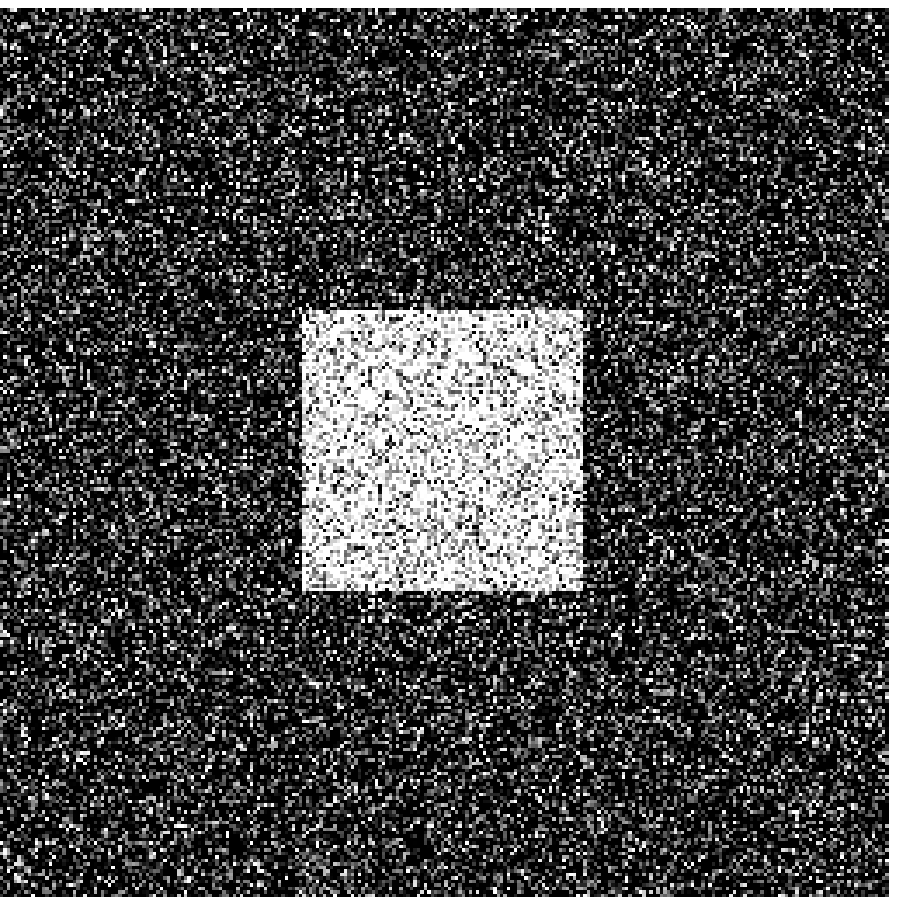}
\includegraphics[width=0.21\textwidth]{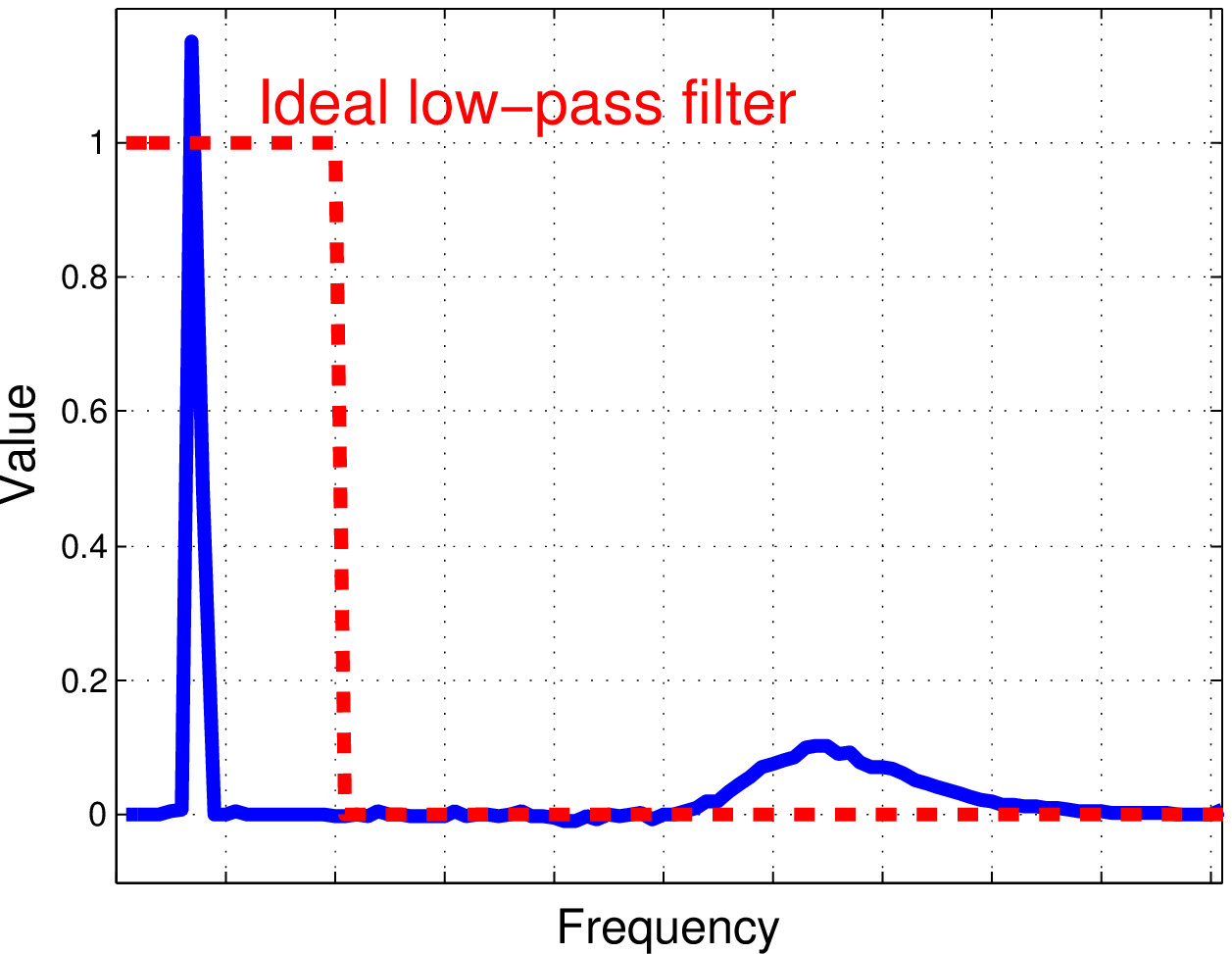}
\includegraphics[width=0.16\textwidth]{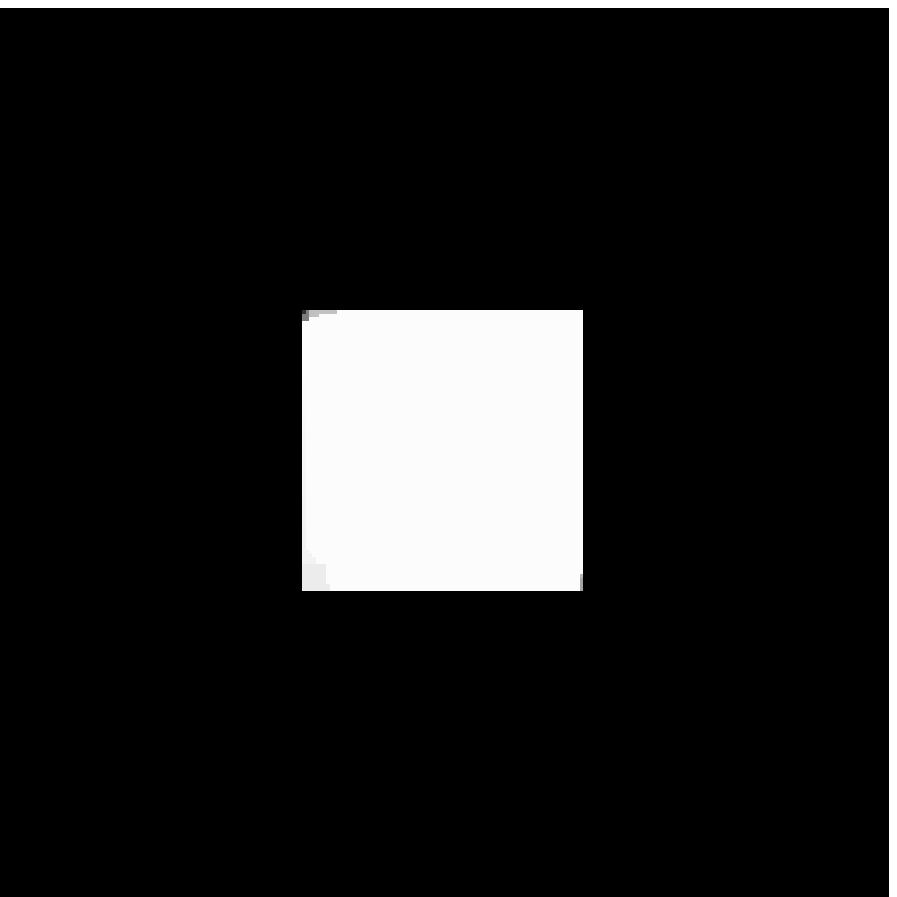}
  \caption{Example for perfect separation of noise and signal via anisotropic TV regularization in the framework of nonlinear spectral decompositions using the inverse scale space flow. From left to right: Clean image, corresponding spectrum of the clean image, noisy image, spectrum of the noisy image with an ideal low pass filter to separate noise and signal, reconstructed image after applying the ideal low pass filter. We used the definition $S_3$ in this illustration of the spectrum.}
  \label{fig:idealLowPass}
\end{figure*}

In \cite{learnedSpectralFilters} a first approach to learning optimal denoising filters on a training data set of natural images demonstrated promising results. In particular, it was shown that optimal filters neither had the shape of ideal low pass filters nor of the filter arising from evolving the gradient flow.

In future research different regularizations for separating noise and signals in a spectral framework will be investigated. Moreover, the proposed spectral framework allows to filter more general types of noise, i.e. filters are not limited to considering only the highest frequencies as noise. Finally, the complete absence of high frequency components may lead to an unnatural appearance of the images, such that the inclusion of some (damped) high frequency components may improve the visual quality despite possibly containing noise. Figure \ref{fig:highFrequenInclusion} supports this conjecture by a preliminary example.
\begin{figure*}
\begin{center}
\includegraphics[width=0.3\textwidth]{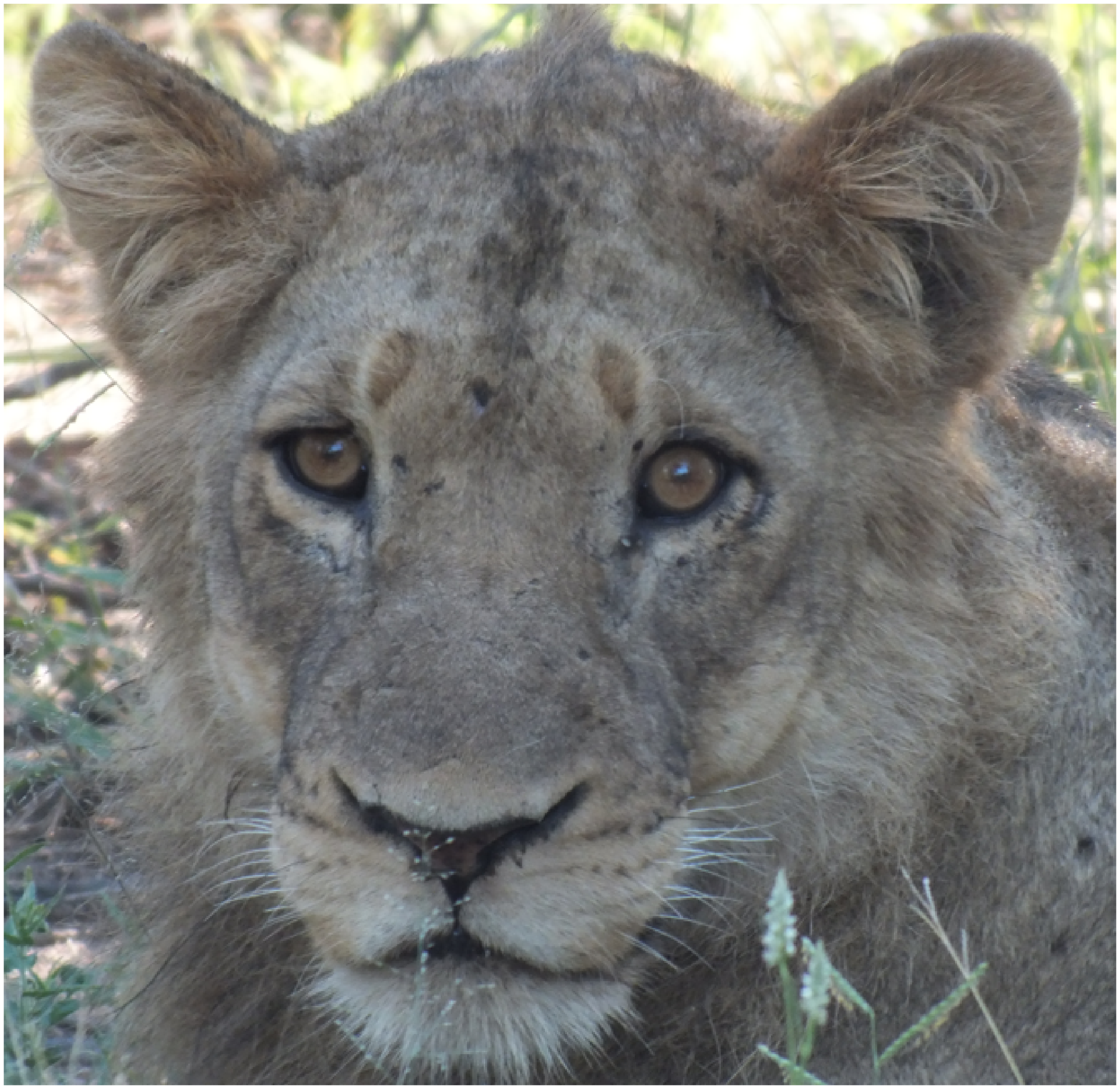} \hspace{0.3cm}
\includegraphics[width=0.3\textwidth]{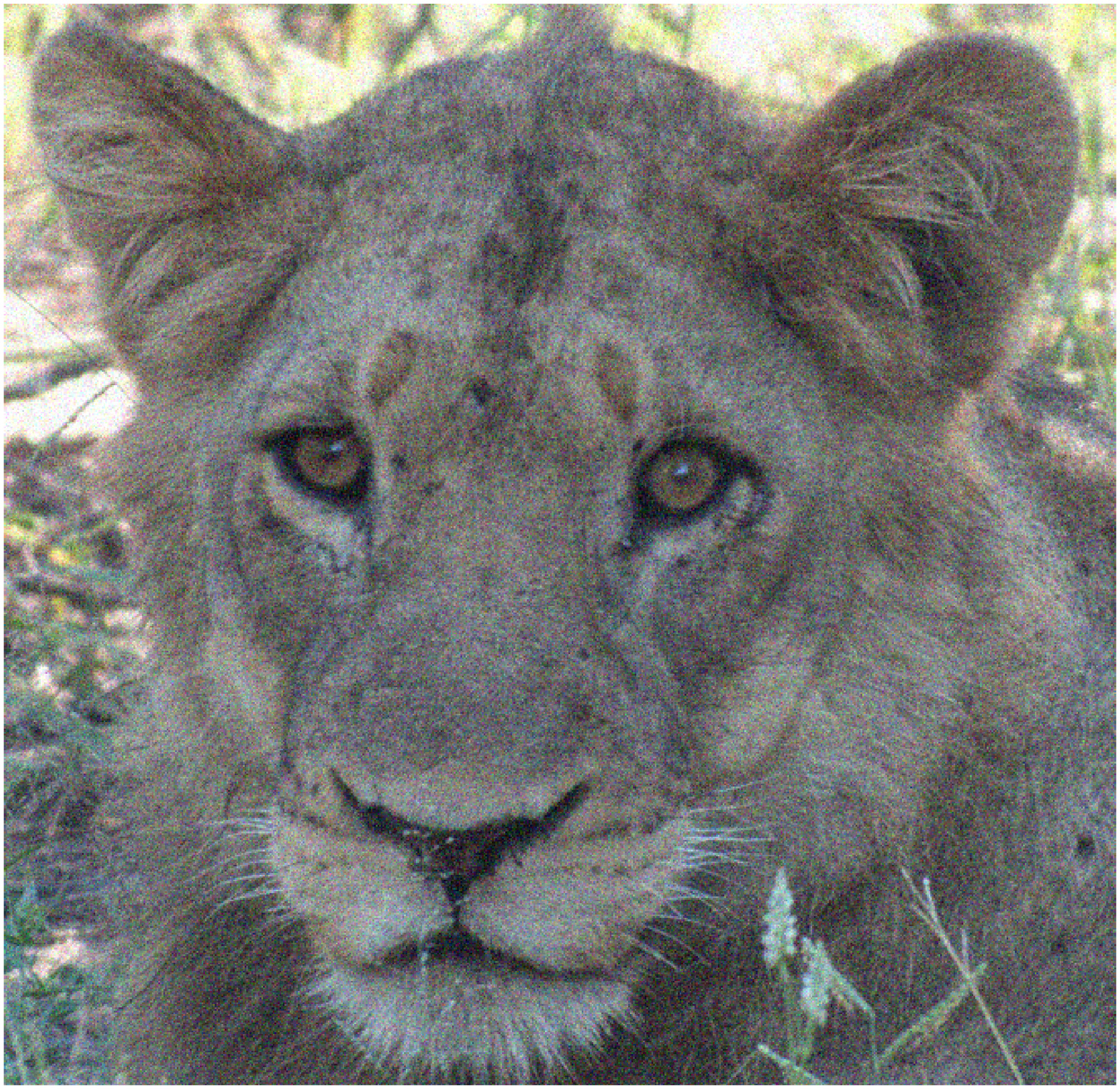}\\
\vspace{0.3cm}
\includegraphics[width=0.3\textwidth]{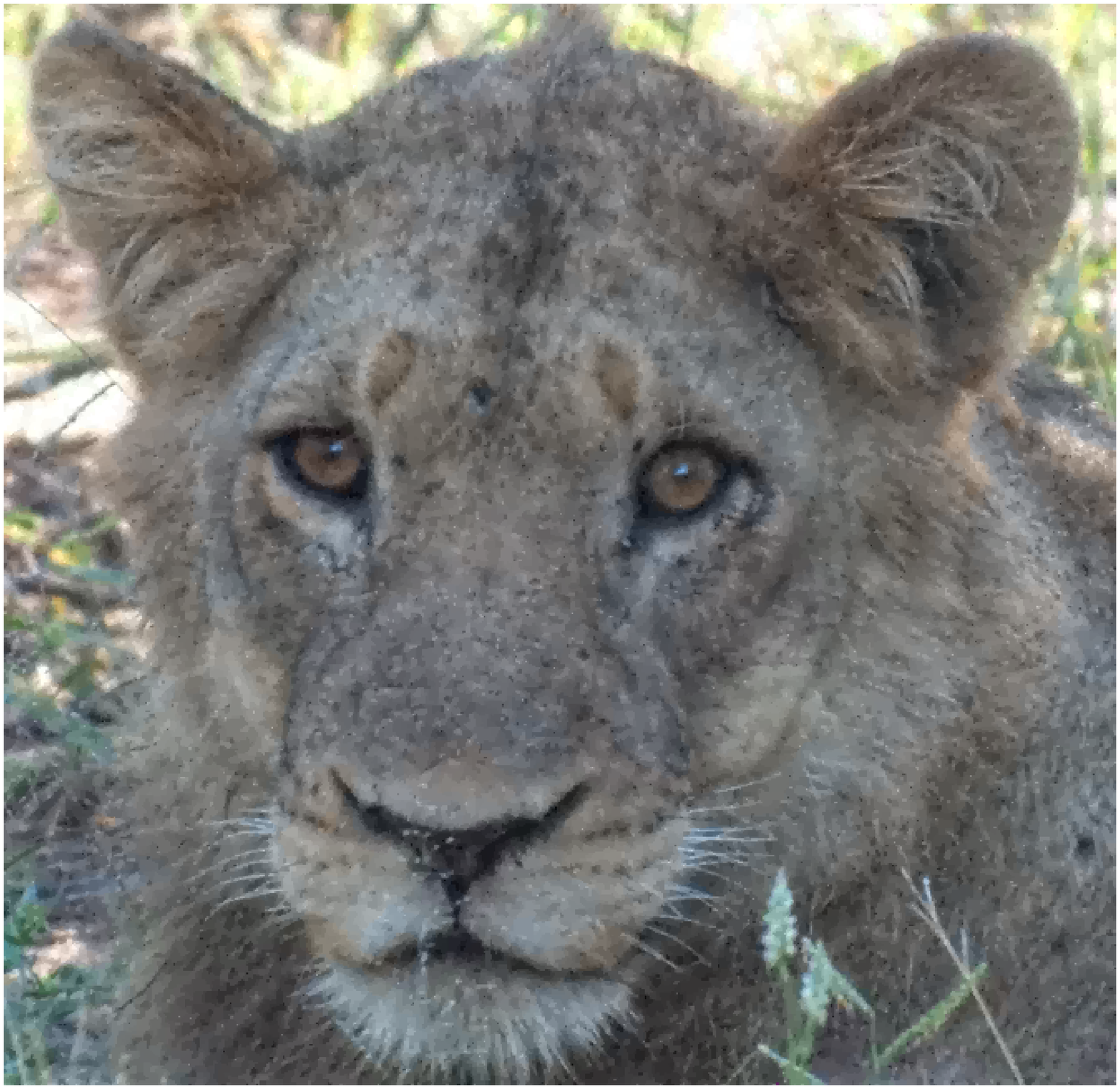}
\includegraphics[width=0.3\textwidth]{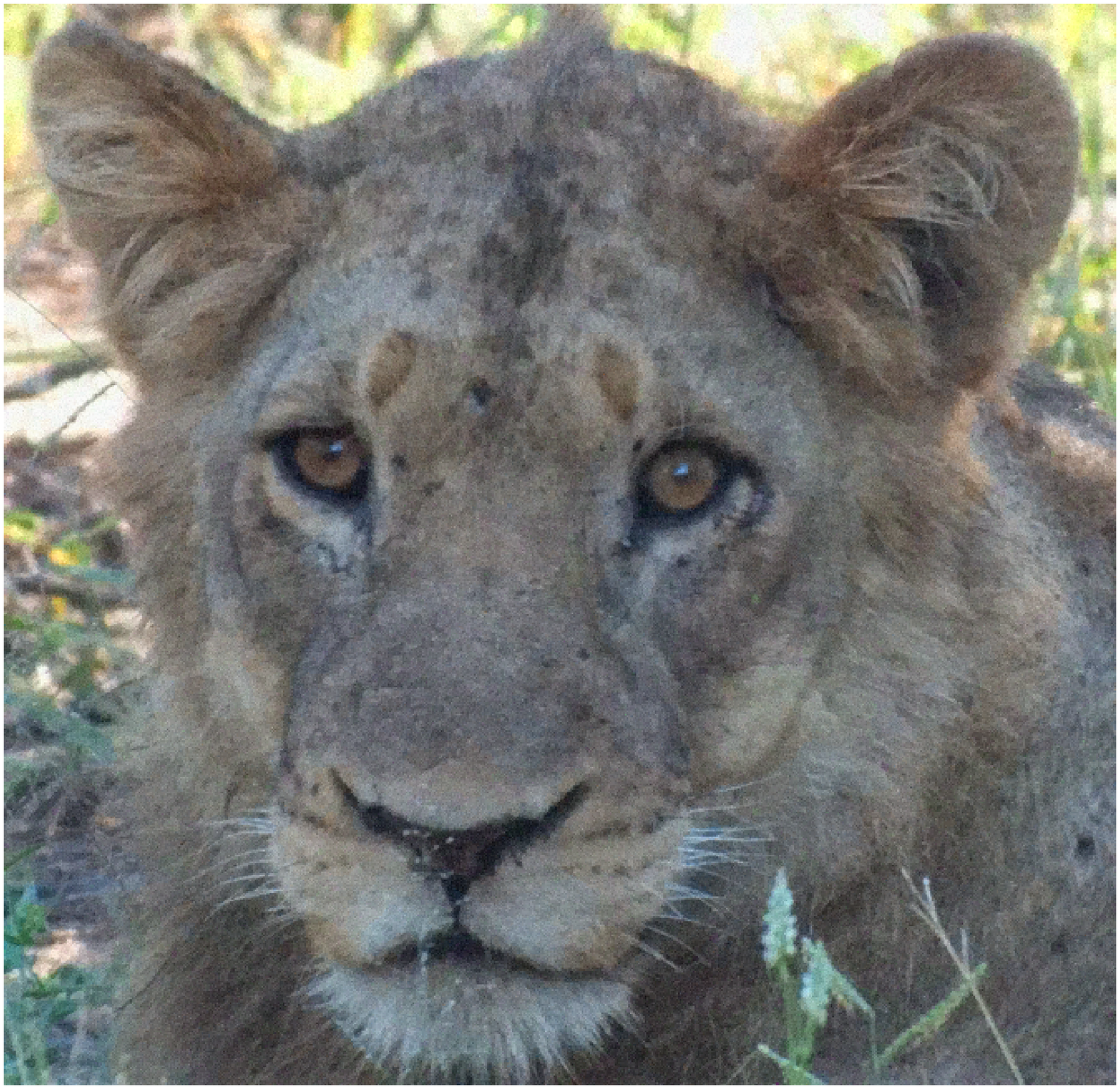}
\includegraphics[width=0.32\textwidth]{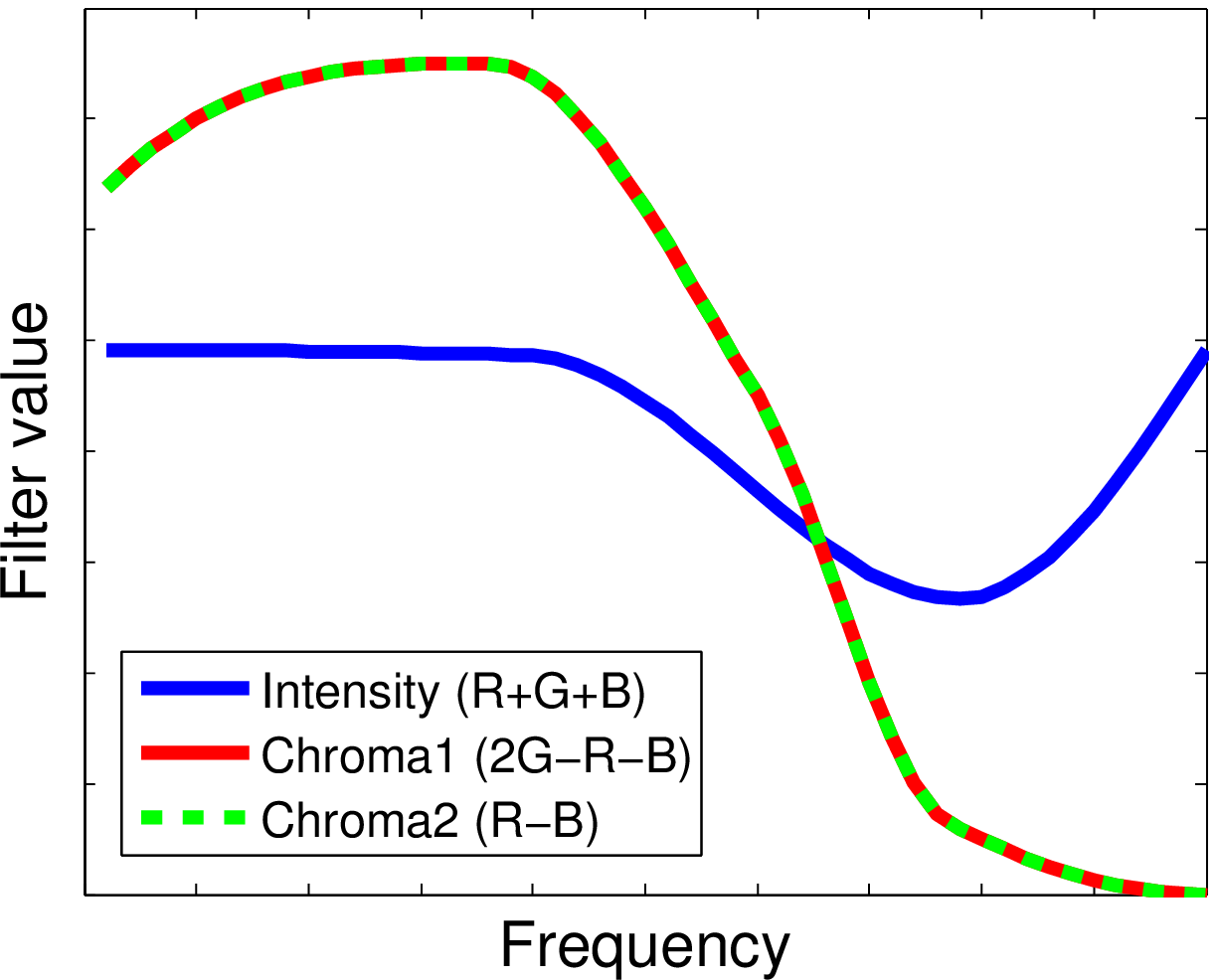}
\end{center}
  \caption{Example for denoising natural images: The original image (upper left) contains a lot of texture. If one tries to denoise the upper right image with an ideal low pass filter (using TV in a color space that decouples intensities and chromaticities), we obtain the lower left image. Since some of the high frequencies are completely suppressed, the image looks unnatural. A designed filter as shown in the lower right can produce more visually pleasing results (lower middle image). This example was produced using the ISS framework such that the filters are shown in frequency representation. Note that the suppression of color artifacts seems to be significantly more important than the suppression of oscillations.}
  \label{fig:highFrequenInclusion}
\end{figure*}

\subsection{Texture processing}
One can view the spectral representation as an extension to infinite dimensions of multiscale approaches for texture decomposition, such as \cite{Tadmor_vese,gilles}. In this sense $\phi(t)$ of \eqref{eq:phi} 
is an infinitesimal textural element (which goes in a continuous manner from ``texture'' to ``structure'', depending on $t$).
A rather straightforward procedure is therefore to analyze the spectrum of an image and either manually or automatically select
integration bands that correspond to meaningful textural parts in the image. This was done in \cite{Horesh_Gilboa_orientation15},
where a multiscale orientation descriptor, based on Gabor filters, was constructed. This yields for each pixel a multi-valued orientation field, which
is more informative for analyzing complex textures. See Fig. \ref{fig:barbara2_tex} as an example of such a decomposition of the Barbara image.

\begin{figure*}[htb]
\begin{tabular}{ccc}
\subfloat[Input]{\includegraphics[width=30mm]{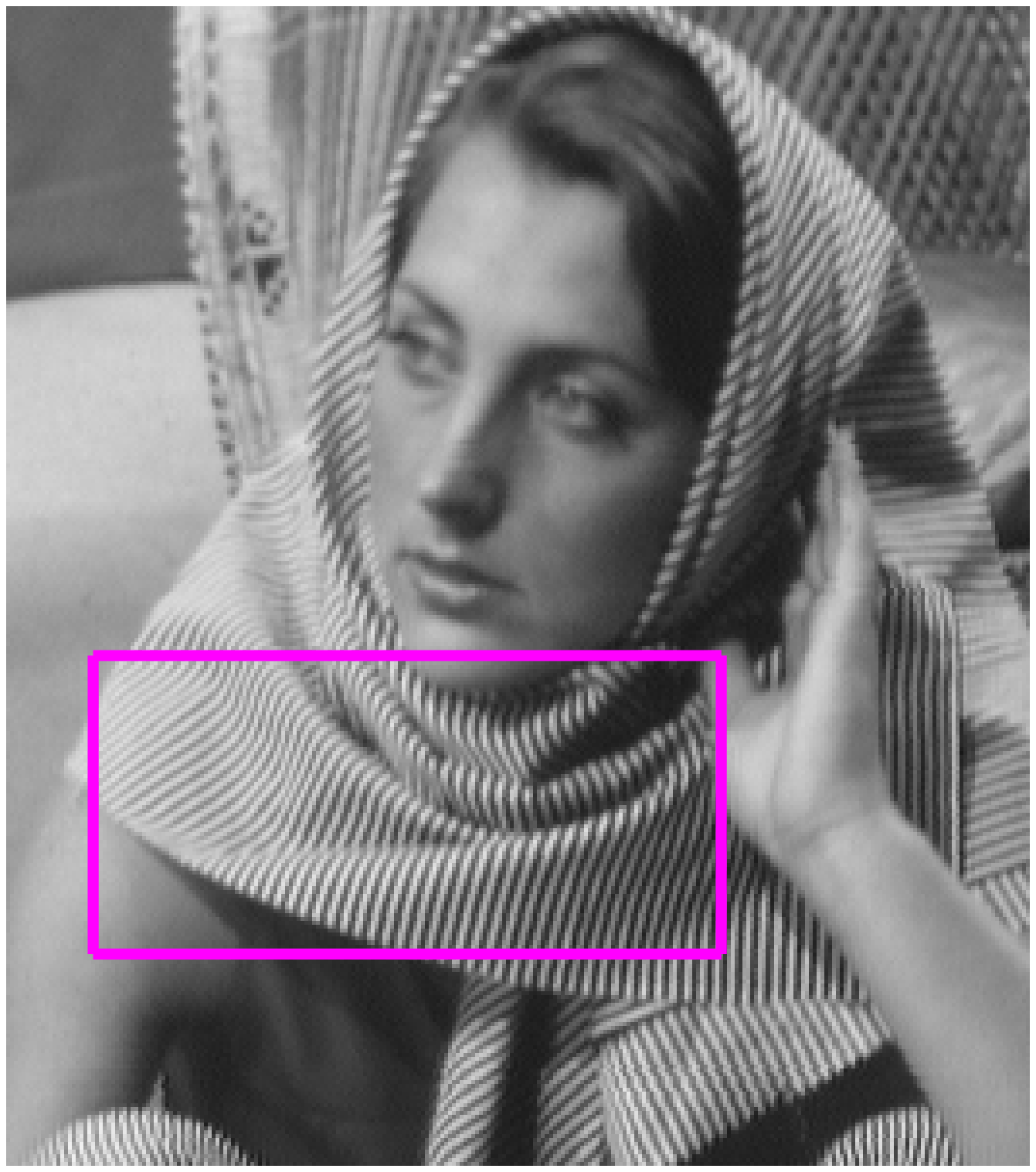}}&
\subfloat[Spectrum]{\includegraphics[width=50mm]{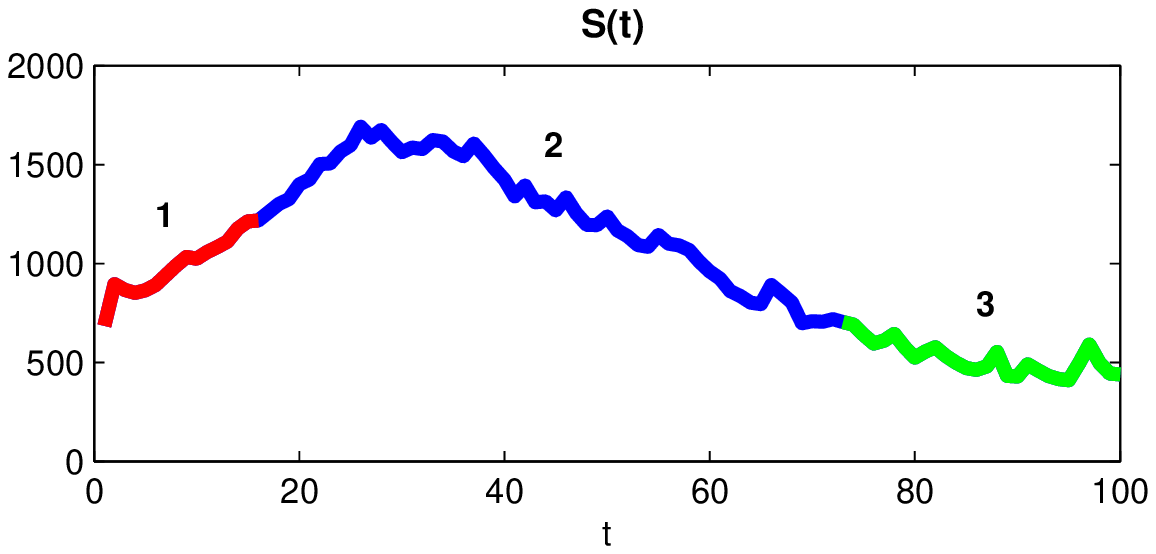}}\\
\subfloat[Scale 1]{\includegraphics[width=50mm]{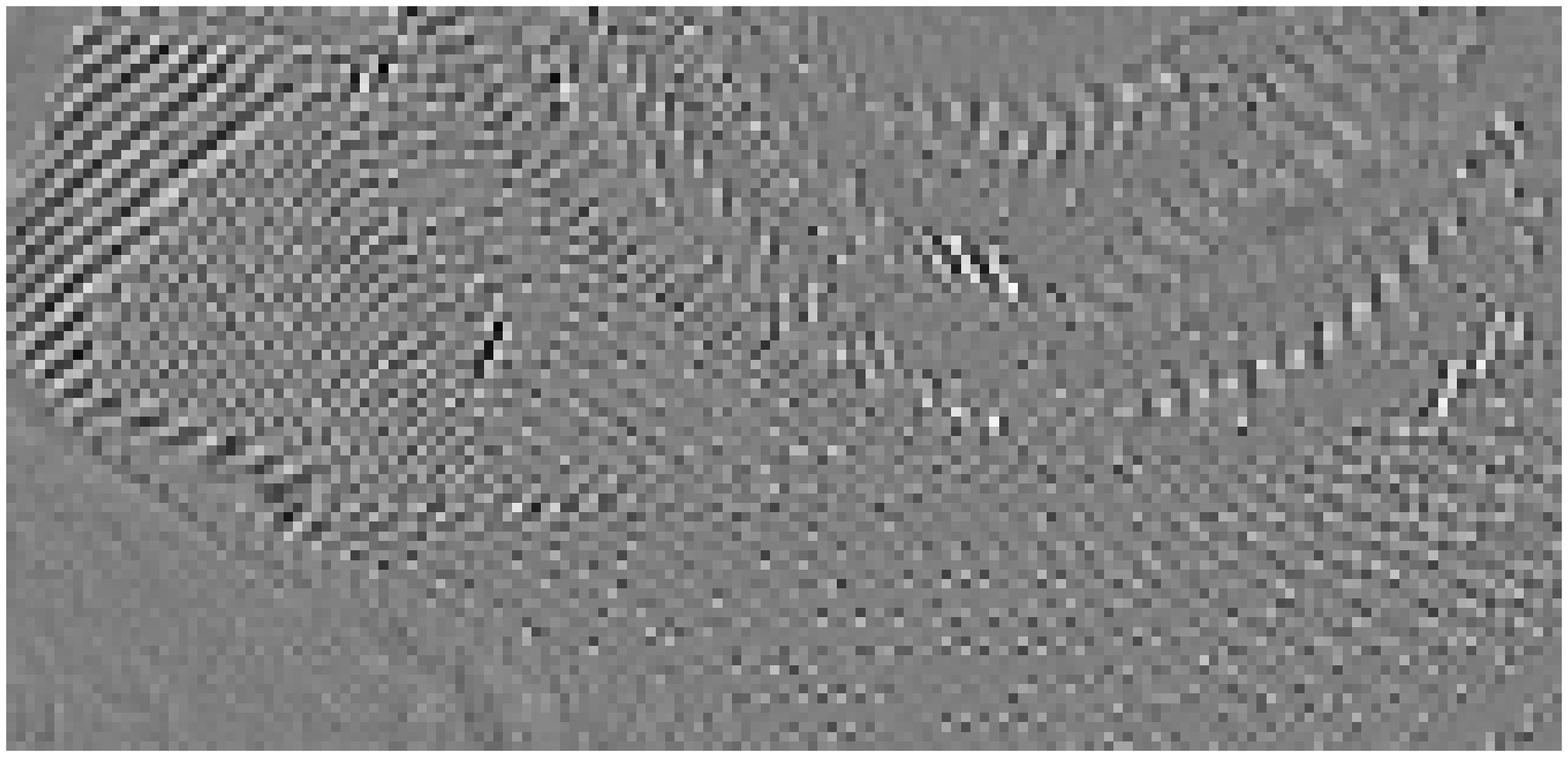}} &
\subfloat[Scale 2]{\includegraphics[width=50mm]{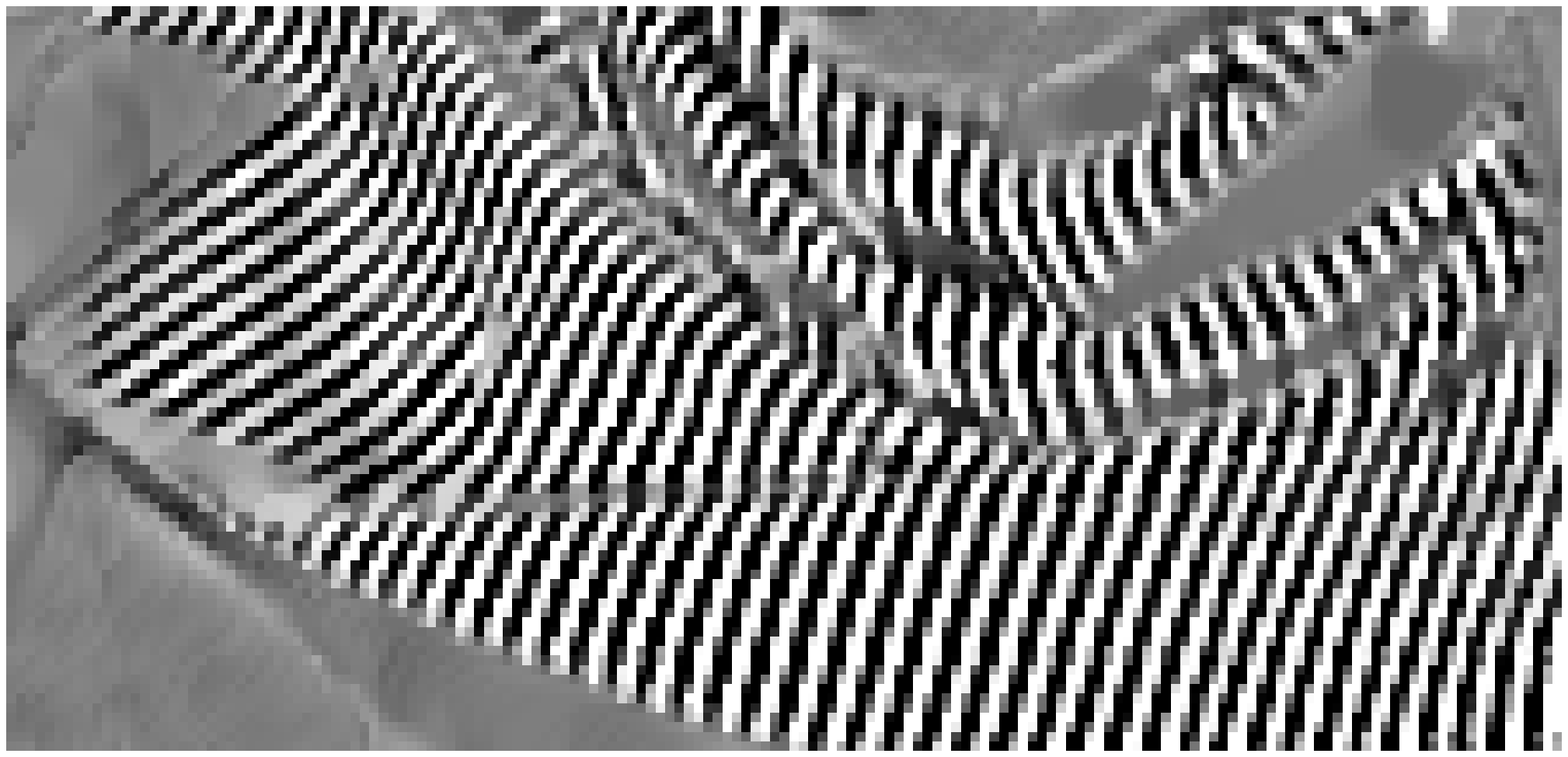}}  &
\subfloat[Scale 3]{\includegraphics[width=50mm]{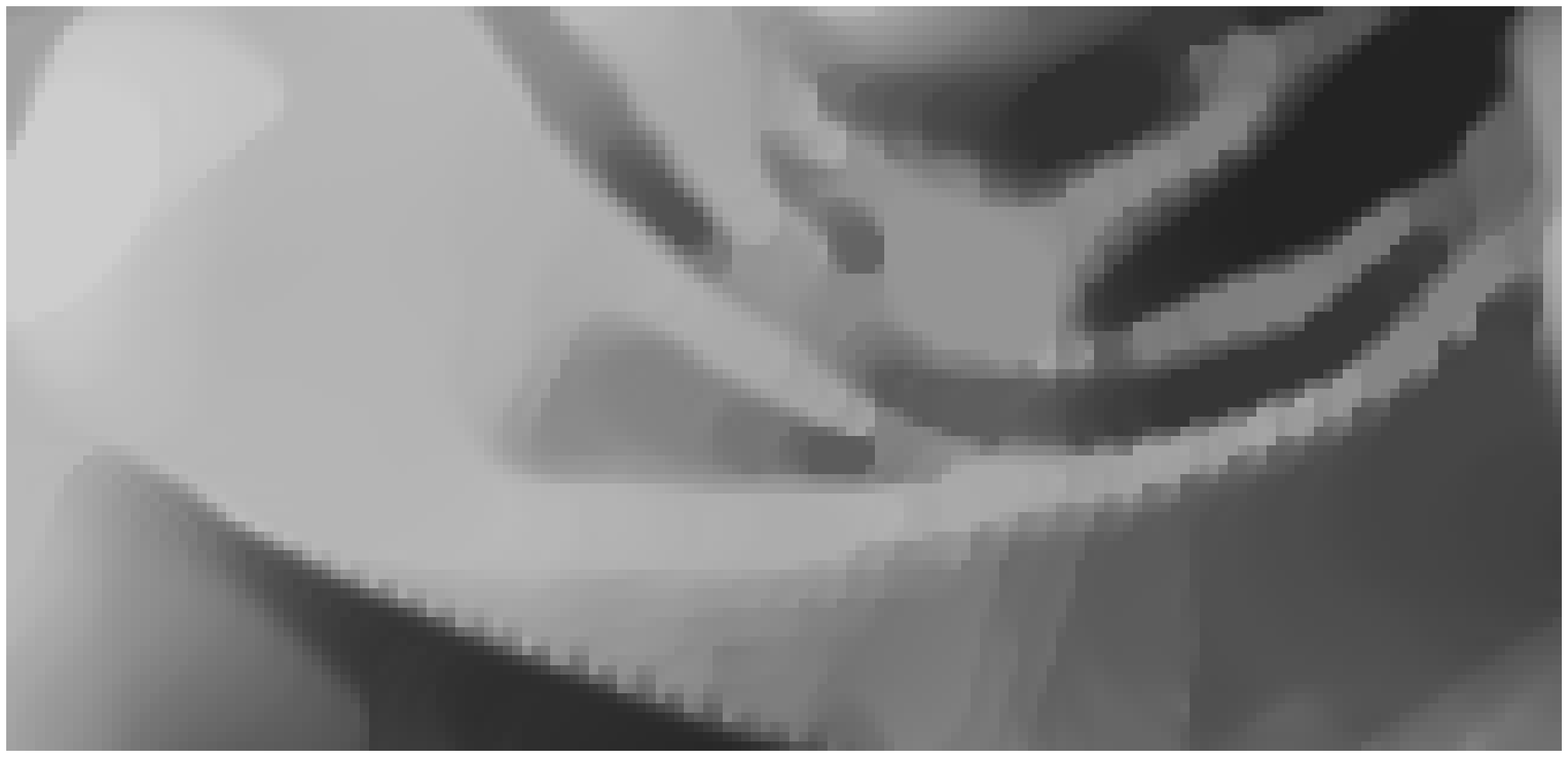}} \\
\subfloat[Orientation 1]{\includegraphics[width=50mm]{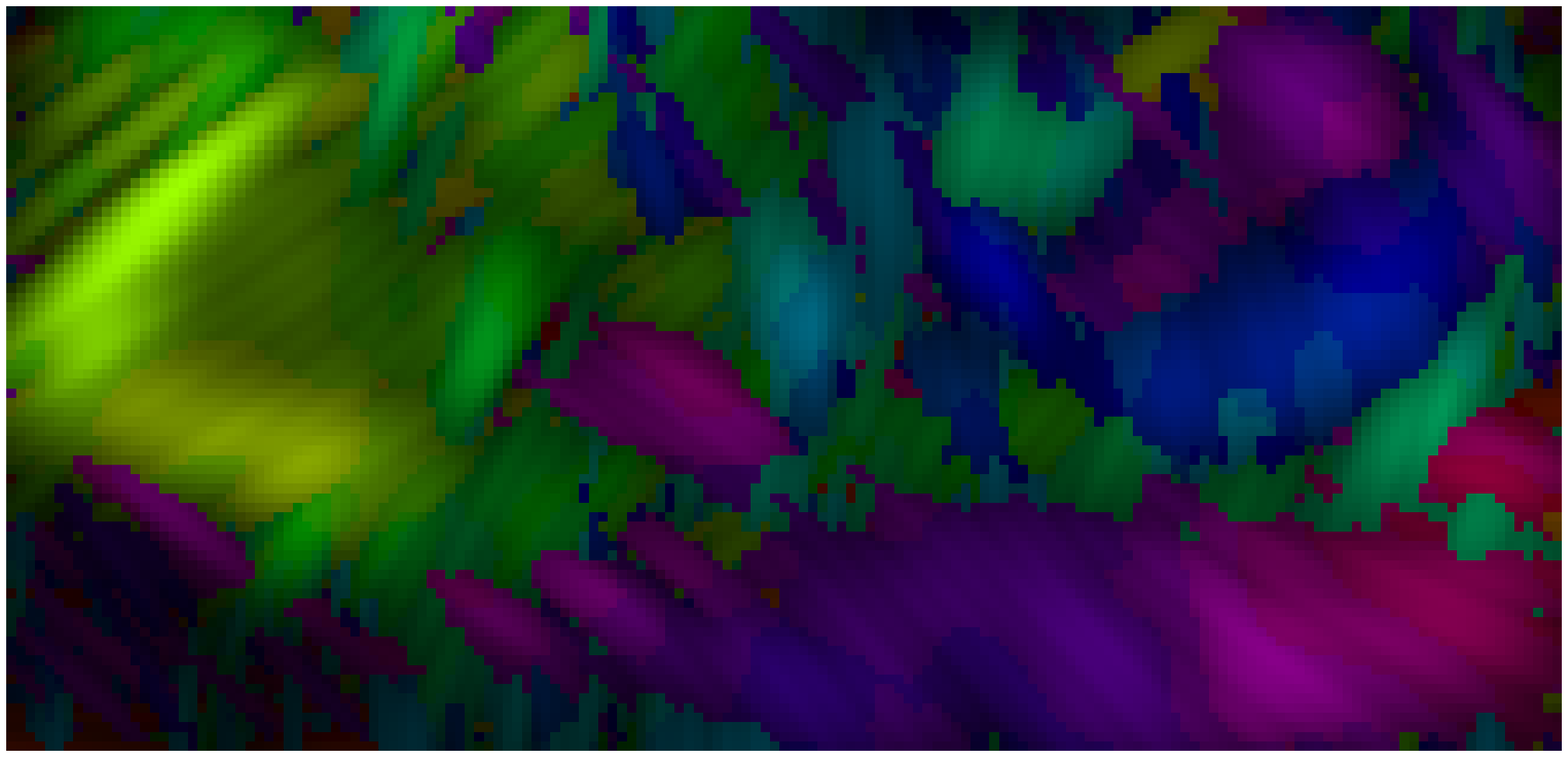}} &
\subfloat[Orientation 2]{\includegraphics[width=50mm]{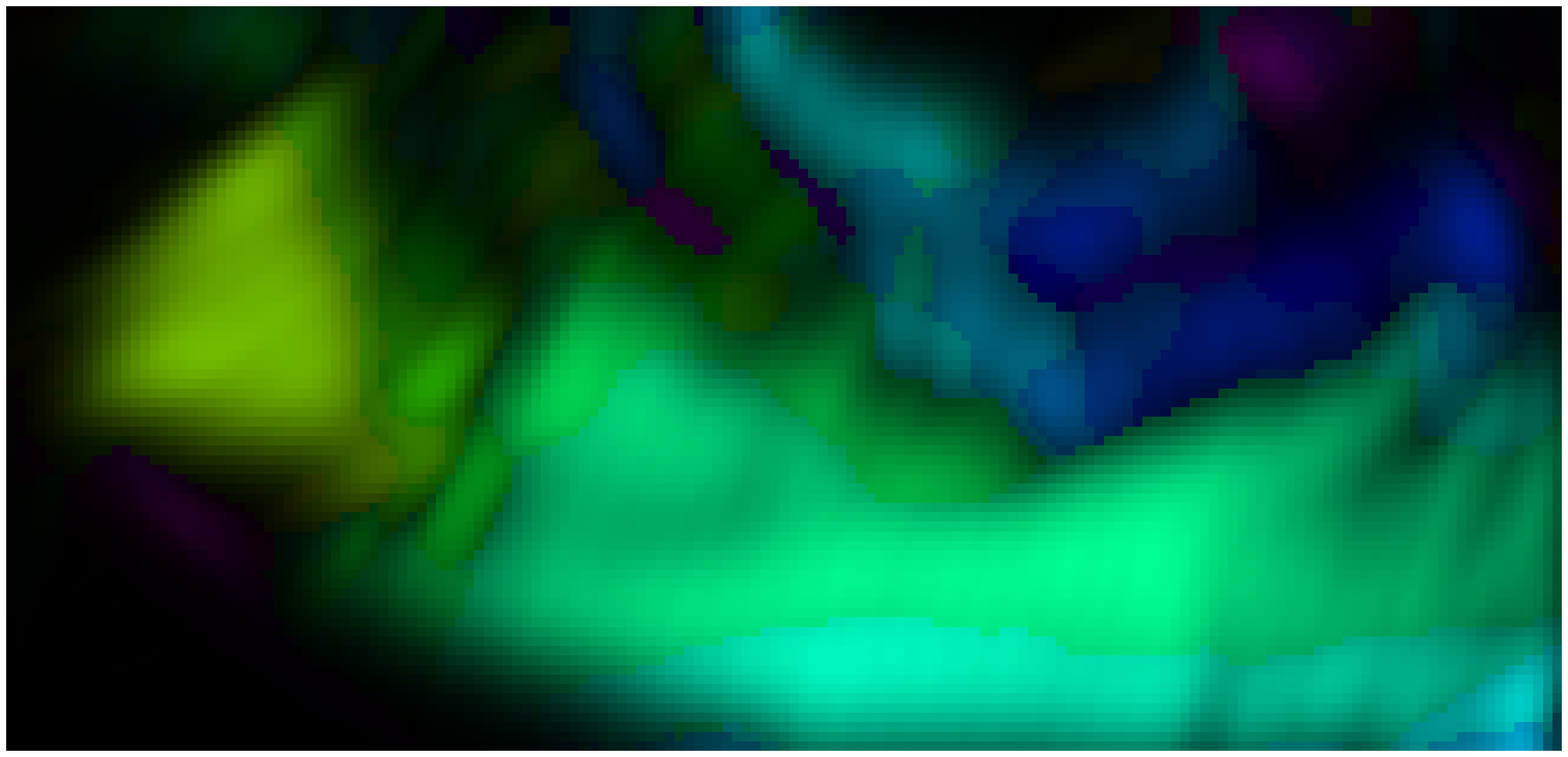}}  &
\subfloat[Orientation 3]{\includegraphics[width=50mm]{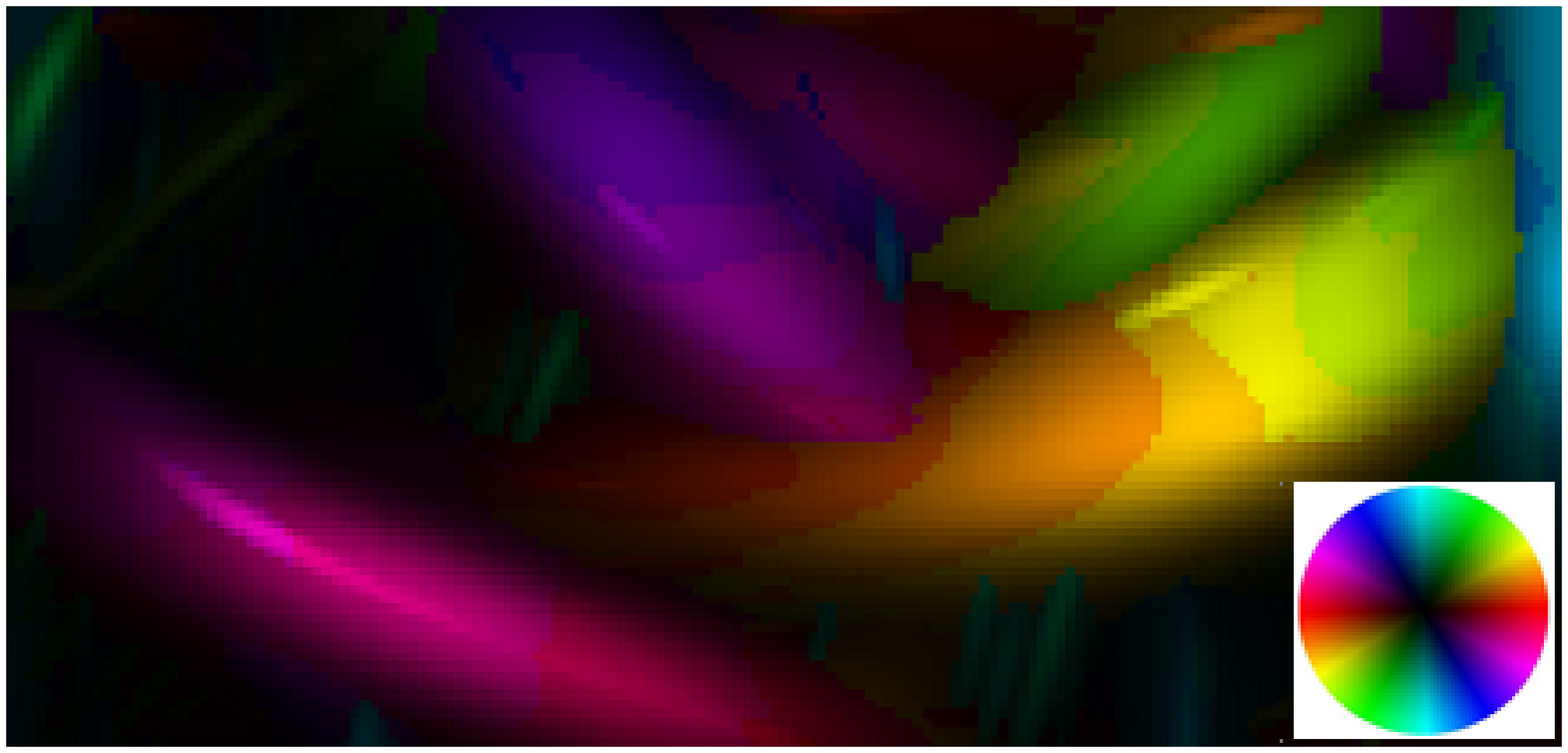}} \\
\end{tabular}
\caption {{Multiscale orientation analysis based on spectral decomposition.} (a) Barbara image, (b) TV spectrum of the image with separated scales marked in different colors. (c-e) Multiscale decomposition, (f-h) Corresponding orientation maps.}
\label{fig:barbara2_tex}
\end{figure*}

\begin{figure*}
\centering
\begin{tabular}{ccc}
\subfloat[Input image]{\includegraphics[width=40mm]{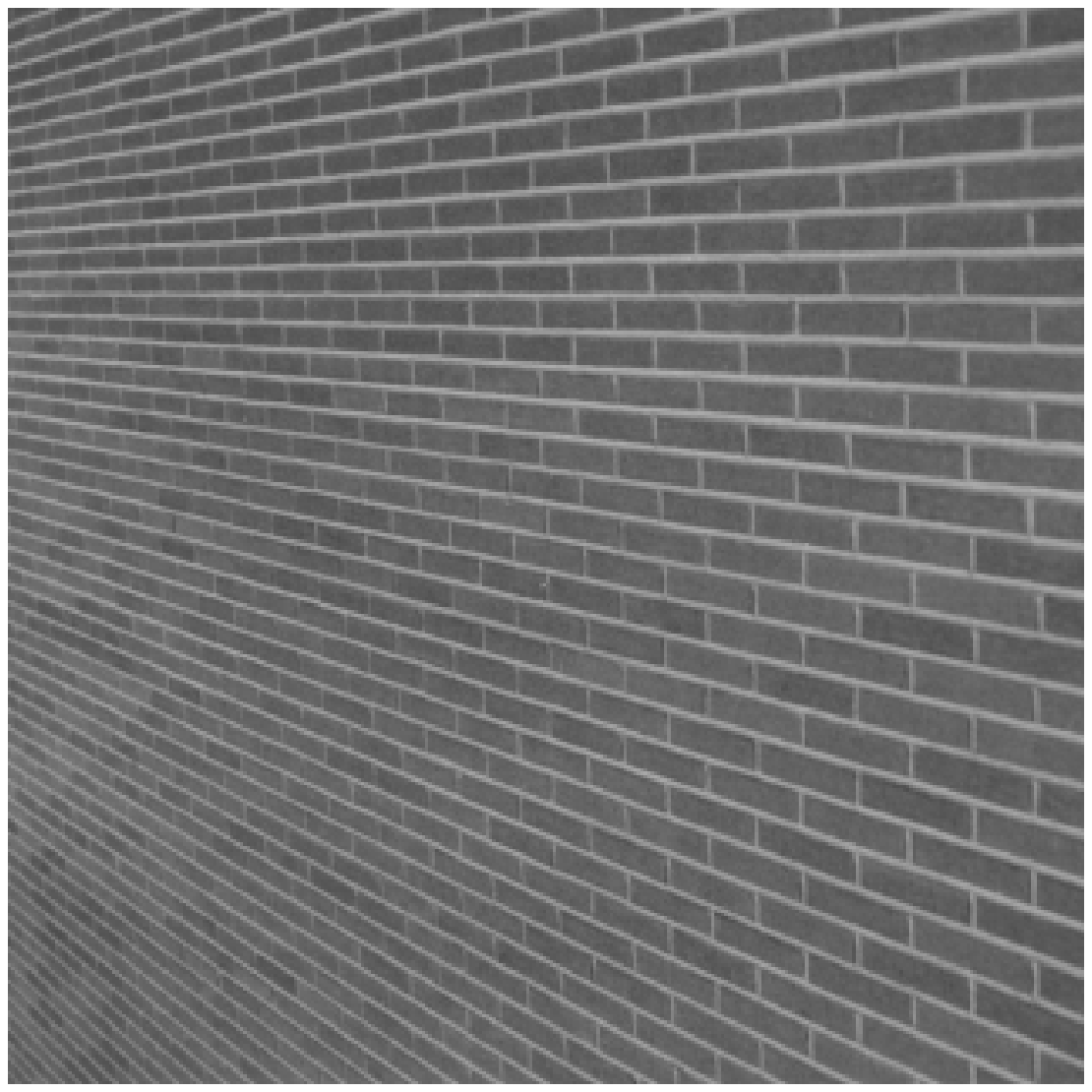}} &
\subfloat[Maximal $\phi$]{\includegraphics[width=60mm]{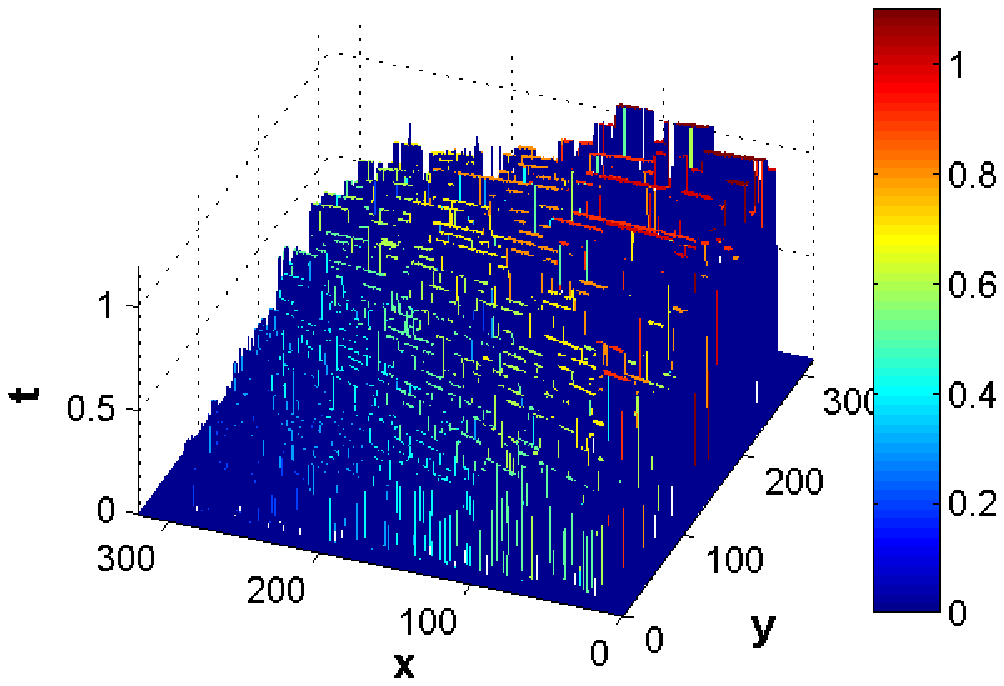}} &
\subfloat[Texture surface]{\includegraphics[width=60mm]{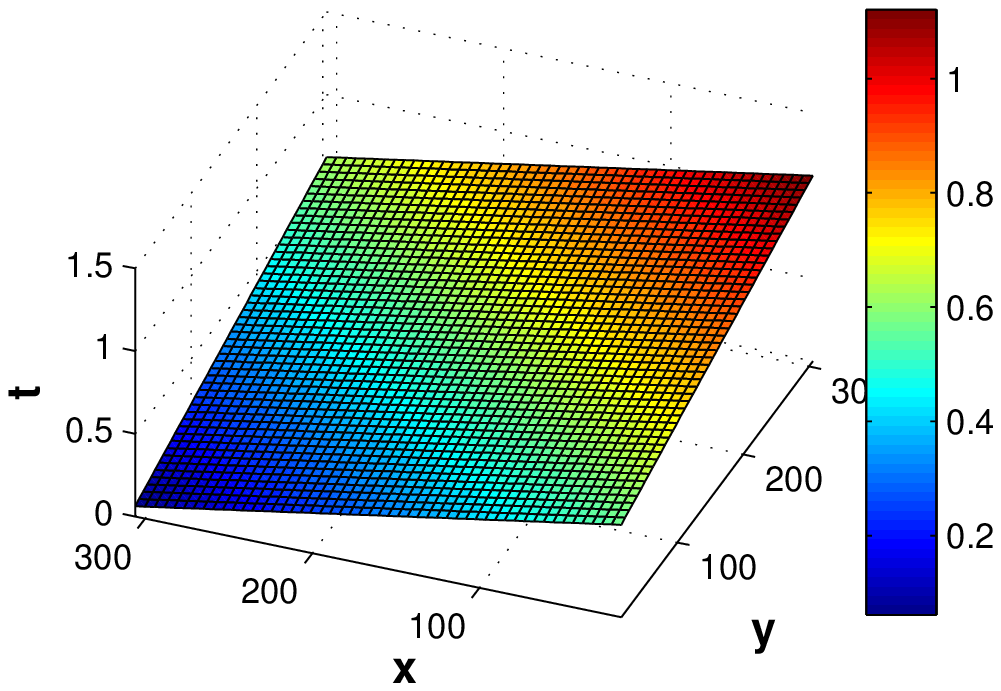}} \\
\end{tabular}
\caption{Spectral analysis of a wall image, from \cite{Horesh_thesis}. (a) Input image of a brick wall (b) $S(x)=\max_t \phi(t;x)$. (c) Approximation of $S(x)$ by a plain.}
\label{fig:wall_spec}
\end{figure*}

\begin{figure*}
\centering
\begin{tabular}{cccc}
\subfloat[Input image]{\includegraphics[width=40mm]{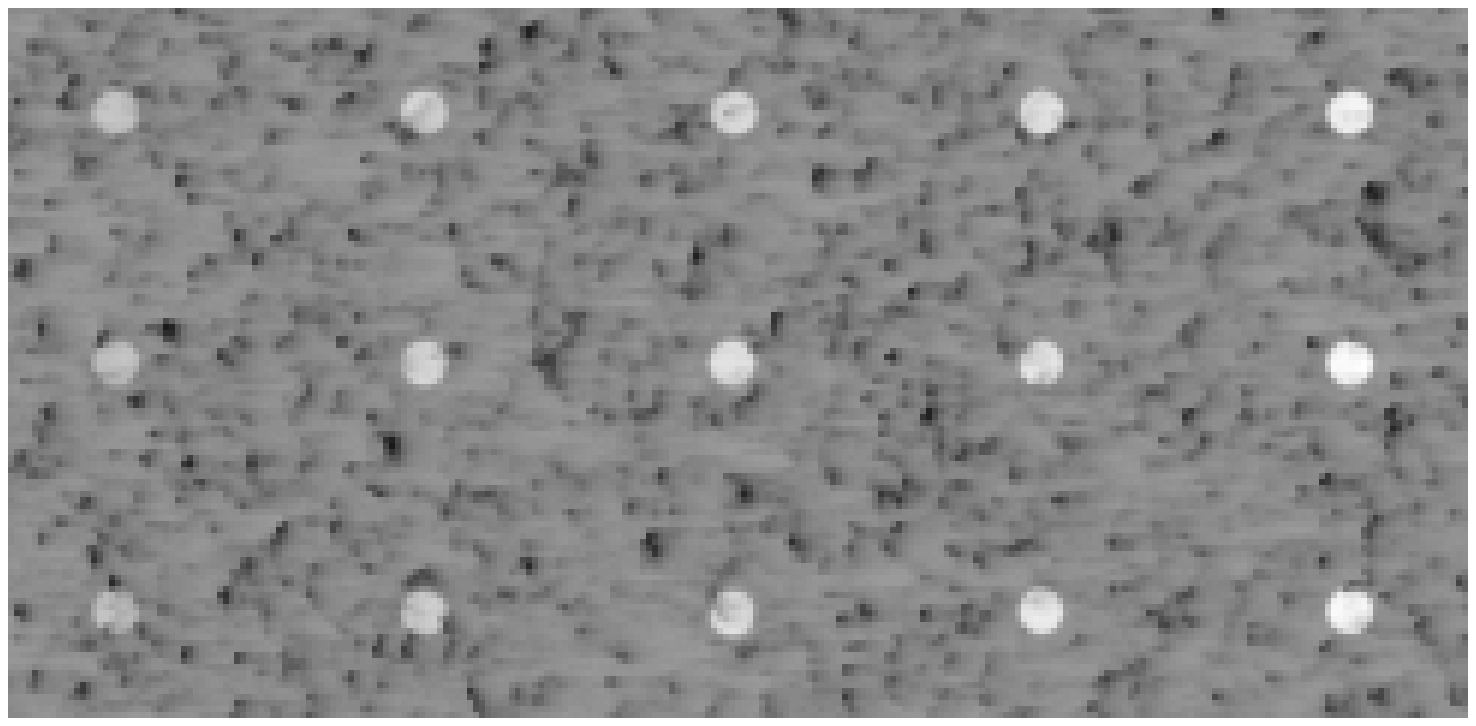}} &
\subfloat[Maximal $\phi$]{\includegraphics[width=40mm]{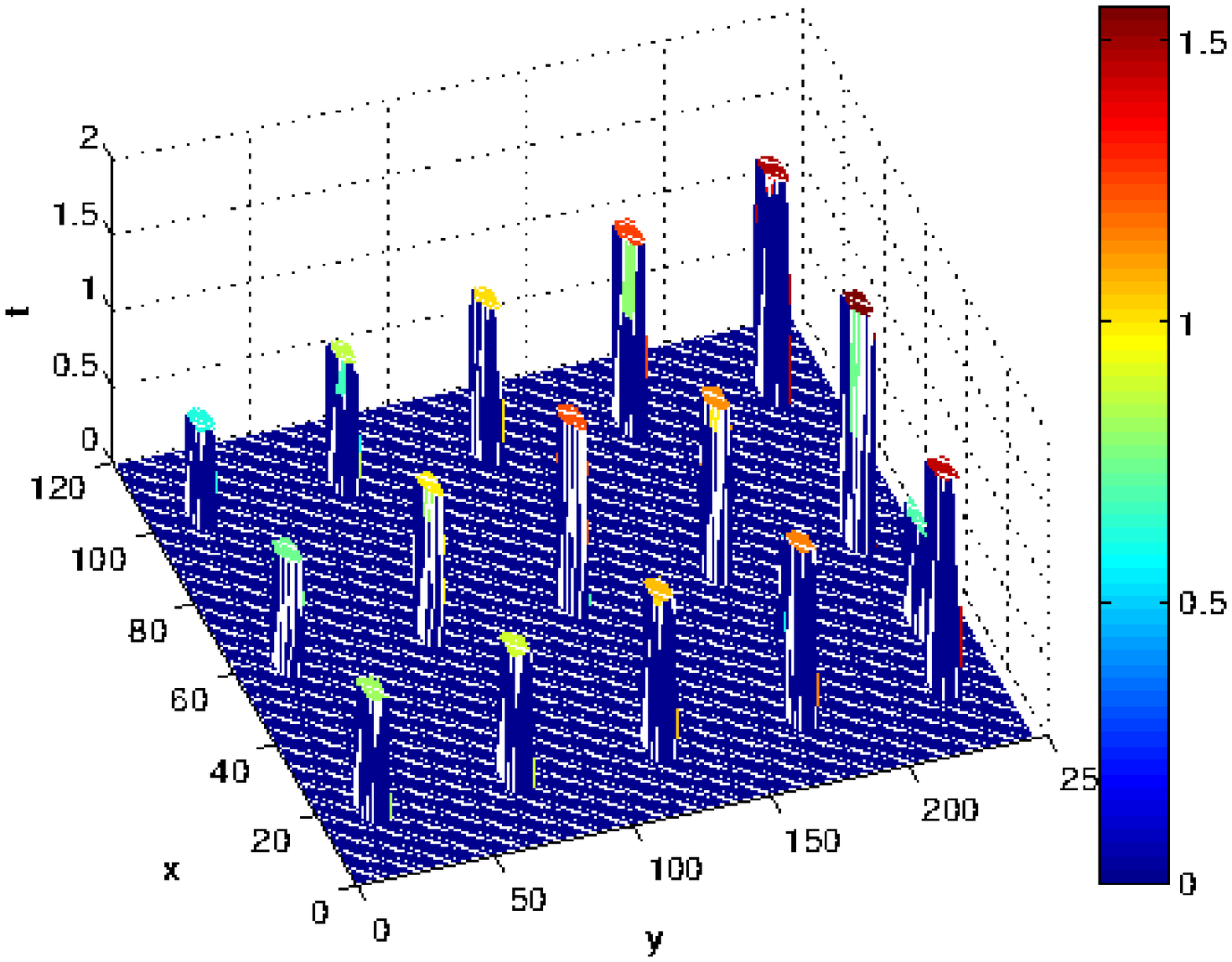}} &
\subfloat[Separation band]{\includegraphics[width=40mm]{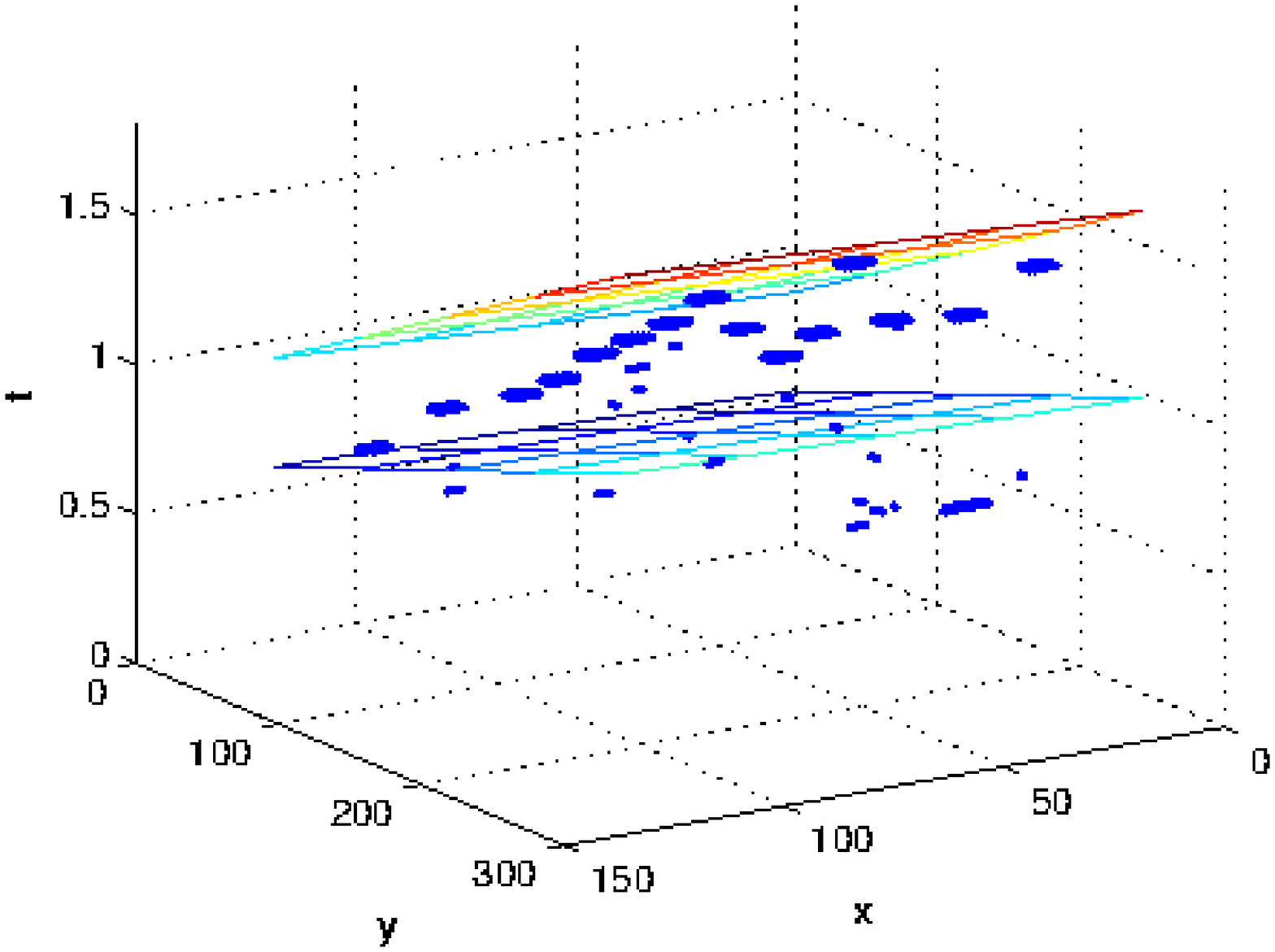}} \\
\subfloat[Spectral band separation \cite{HoreshGilboa_submitted}, layer 1 ]{\includegraphics[width=40mm]{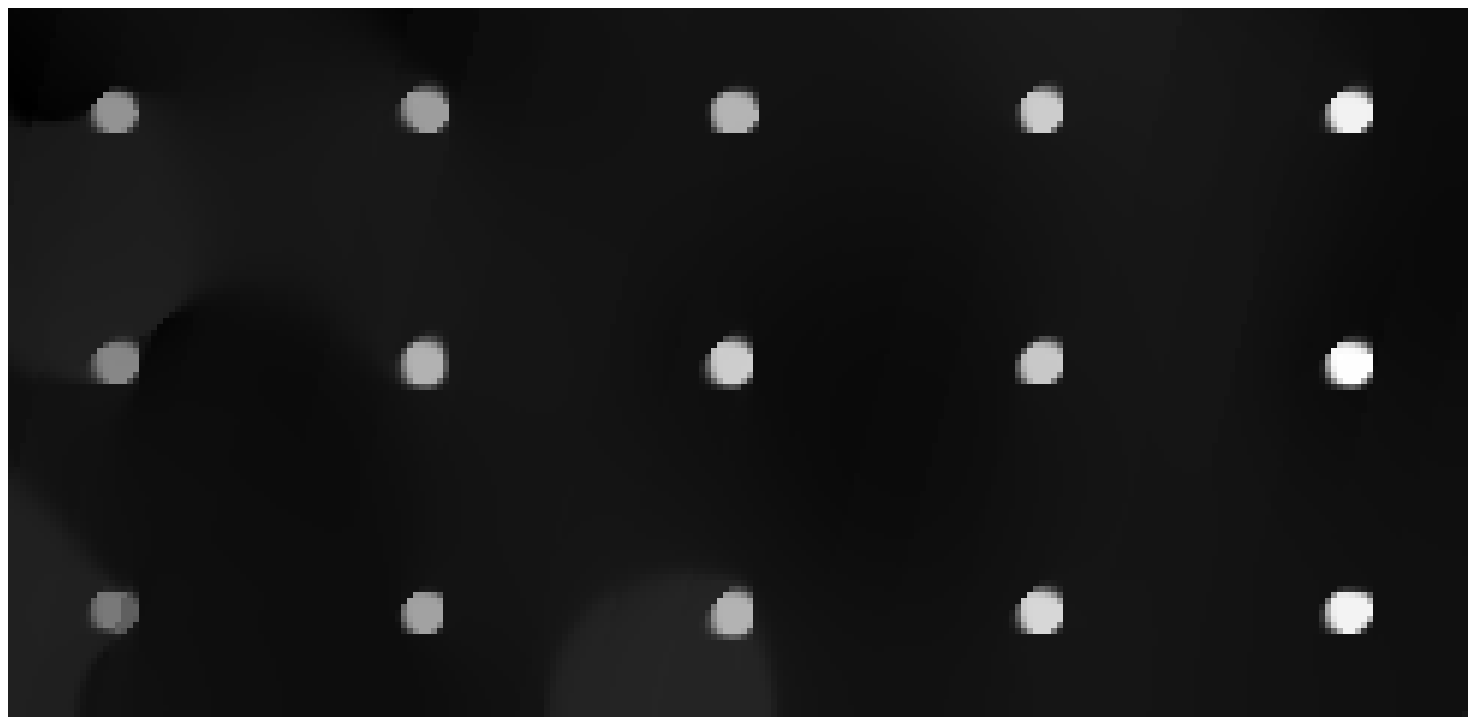}} &
\subfloat[Spectral band separation, layer 2]{\includegraphics[width=40mm]{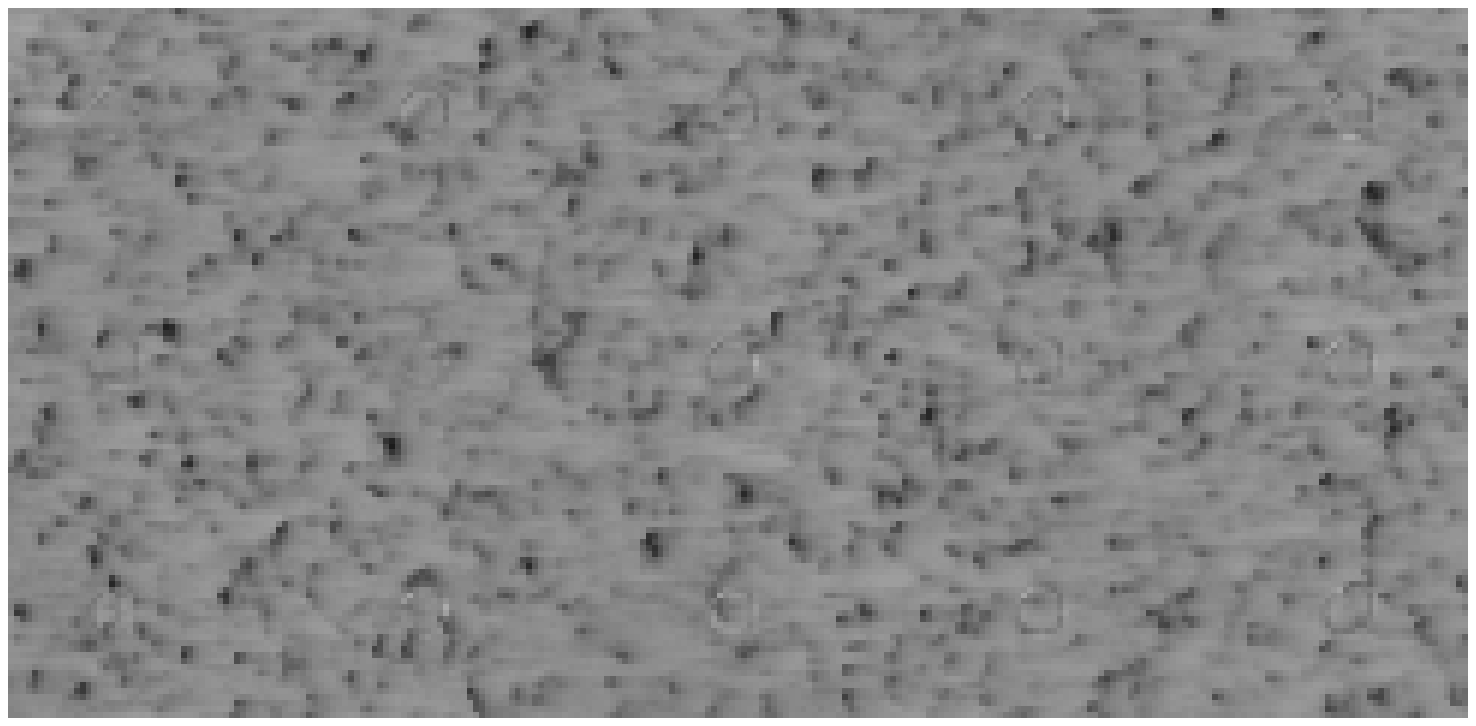}} &
\subfloat[TV-G \cite{Aujol[3]}, layer 1]{\includegraphics[width=40mm]{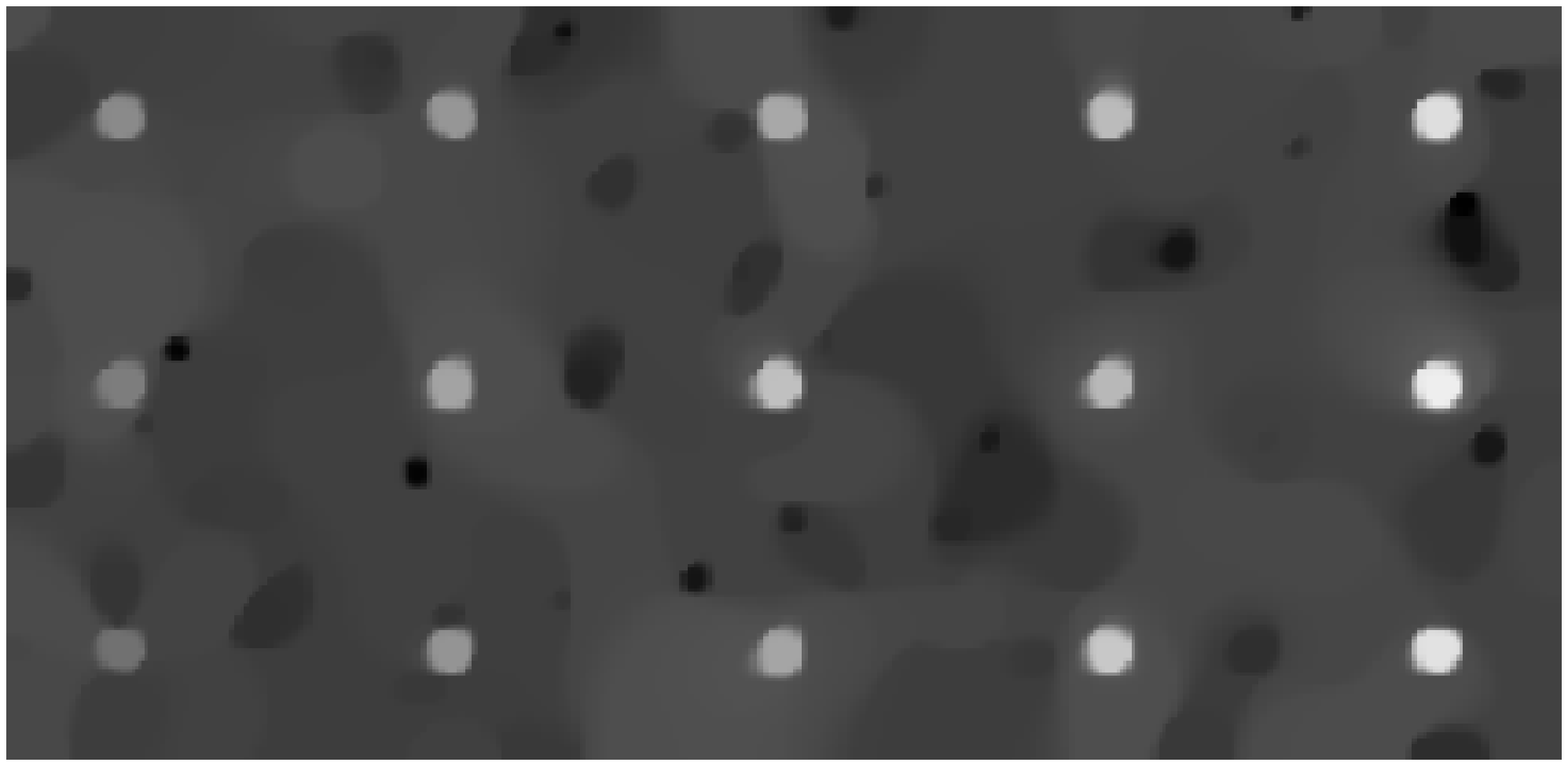}} &
\subfloat[TV-G, layer 2]{\includegraphics[width=40mm]{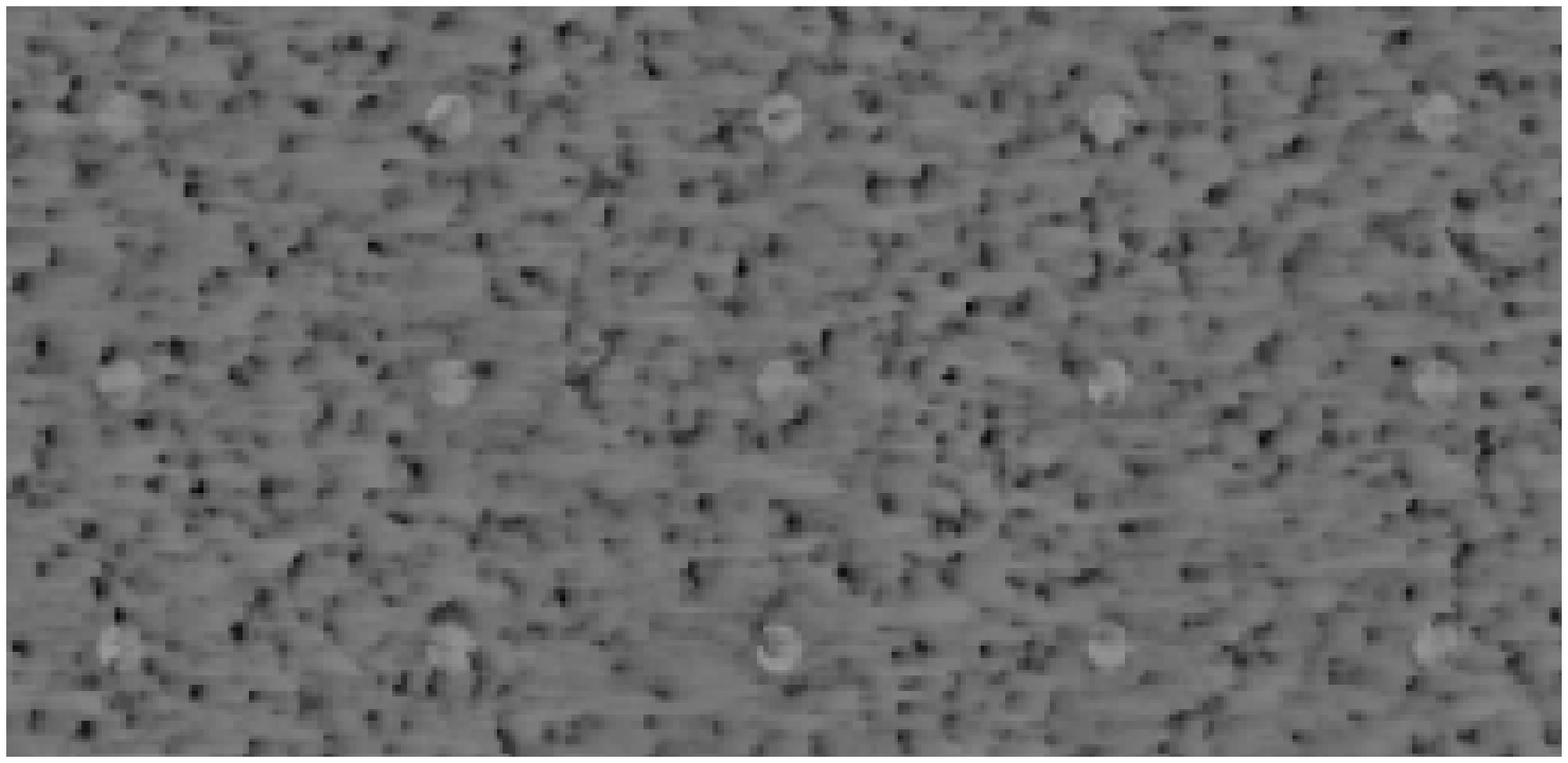}}
\end{tabular}
\caption[Decomposition by spectral analysis] {{Decomposition by a separation band in the spectral domain.} (a) Input image, (b) The maximal $\phi$ response, (c) The separation band, (d-e) spectral decomposition, (f-g) TV-G decomposition. Taken from \cite{HoreshGilboa_submitted}.  }
\label{fig:text1504_spec}
\end{figure*}

In \cite{HoreshGilboa_submitted,Horesh_thesis} a unique way of using the spectral representation was suggested for texture processing.
It was shown that textures with gradually varying pattern-size, pattern-contrast or illumination can be represented by surfaces in the
three dimensional TV transform domain. A spatially varying texture decomposition amounts to estimating
a surface which represents significant maximal response of $\phi$, within a time range $[t_1,t_2]$,
$$\max_{t\in [t_1,t_2]} \{\phi(t;x)\} > \epsilon $$
 for each spatial coordinate $x$.
In Fig. \ref{fig:wall_spec} a wall is shown with gradually varying texture scale and its scale representation in the spectral domain.
A decomposition of structures with gradually varying contrast is shown in Fig. \ref{fig:text1504_spec} where the band-spectral decomposition
(spatially varying scale separation) is compared to the best TV-G separation of \cite{Aujol[3]}.

\section{Discussion and Conclusion}
In this paper we presented the rationale for analyzing one-homogeneous variational problems through a spectral approach.
It was shown that solutions of the generalized nonlinear eigenvalue problem \eqref{eq:ef_problem} are a fundamental part of this analysis.
The quadratic (Dirichlet) energy yields Fourier frequencies as solutions, and thus eigenfunctions of one-homogeneous functionals can be viewed as
new non-linear extensions of classical frequencies.

Analogies to the Fourier case, to wavelets and to dictionary representations were drawn.
However, the theory is only beginning to be formed and there are many open theoretical problems, a few examples are:
\begin{enumerate}
\item More exact relations of the one-homogeneous spectral representations to basis and frames representations.
\item Understanding the difference between the gradient-flow, variational and inverse-scale-space representations.
\item Consistency of the decomposition: If we filter the spectral decomposition of some input data $f$, e.g. by an ideal low pass filter at frequency $T$, and apply the spectral decomposition again, will we see any frequencies greater than $T$?
\item Can we show some kind of orthogonality of the spectral representation $\phi$?
\item Spectral analysis of random noise.
\item Can the theory be extended from the one-homogeneous case to the general convex one?
\end{enumerate}

In addition, there are many practical aspects, such as
\begin{enumerate}
\item Learning the regularization on a training set with the goal to separate certain features.
\item Numerical issues, computing $\phi$ in a stable manner (as it involves a second derivative in time), also can one design
better schemes than the ones given here, which are less local and converge fast for any eigenfunction.
\item Additional applications where the new representations can help in better design of variational algorithms.
\end{enumerate}

Some initial applications related to filter design and to texture processing were shown.
This direction seems as a promising line of research which can aid in better understanding of variational processing
and that has a strong potential to provide alternative improved algorithms.

\small{
\bibliography{refs}

}
\end{document}